\documentclass[12pt,a4paper]{article}

\usepackage{enumitem}
\usepackage[english]{babel}
\usepackage{amsmath}
\usepackage{amssymb}
\usepackage{mathrsfs}
\usepackage{tikz}
\usetikzlibrary{cd}
\usetikzlibrary{calc}
\usepackage{graphicx}
\usepackage[english]{babel}
\usepackage{amsfonts,amsthm,amsopn}
\usepackage[utf8]{inputenc}
\usepackage{stmaryrd}
\usepackage[colorlinks=true,allcolors=black]{hyperref}
\usepackage{mathtools}

\DeclareMathOperator{\Br}{Br}
\DeclareMathOperator{\Int}{Int}
\DeclareMathOperator{\NC}{NC}
\DeclareMathOperator{\Ker}{Ker}
\DeclareMathOperator{\len}{len}
\DeclareMathOperator{\Red}{Red}
\DeclareMathOperator{\id}{id}
\DeclareMathOperator{\Art}{Art}
\newcommand\coloneqq{:=}
\newcommand\dapprox{\mathrel{\dot\approx}}
\newcommand{\Z}{\mathbb{Z}}
\newcommand{\C}{\mathbb{C}}

\numberwithin{equation}{section}
\theoremstyle{definition}
\newtheorem{teo}[equation]{Theorem}
\newtheorem{lemma}[equation]{Lemma}
\newtheorem{prop}[equation]{Proposition}
\newtheorem{cor}[equation]{Corollary}
\newtheorem{oss}[equation]{Remark}

\usepackage[margin=1.9cm,top=3cm,bottom=3cm]{geometry}

\usepackage[
backend=biber,
style=alphabetic,
sorting=nyt
]{biblatex}

\title{Free products and the isomorphism between\\standard and dual Artin groups}
\author{Sirio Resteghini\thanks{sirio.resteghini@phd.unipi.it}}
\date{ }
\begin{document}

\maketitle

\abstract{
Given a Coxeter system with a fixed Coxeter element, there is a surjective group morphism $\Psi$ from the standard to the dual Artin groups. We give conditions that are sufficient, necessary or equivalent to $\Psi$ being an isomorphism. In particular, we prove that if the Hurwitz action on the reduced words of any element in the noncrossing partition poset is transitive, and if the Hurwitz action on the reduced words of the Coxeter element has the same stabilizer as essentially the same action viewed in the standard Artin group, then $\Psi$ is an isomorphism. Both of those conditions are already known in some cases, notably in spherical and affine types. We then prove that taking the free (or direct) product of groups that satisfy those two conditions yields another group that, with a suitable Coxeter element, also satisfies them.
}
\tableofcontents

%\newpage

\section{Introduction}
Given a Coxeter system $(W,S)$ with set of reflections $T$ and Coxeter element $h$, one can define the standard Artin group $\Art(W,S)$ and the dual Artin group $\Art^*(W,T,h)$. It is conjectured that these two groups are isomorphic. This conjecture has been proven in the spherical \cite{bessis2002dual}, affine \cite{mccammond2017artin,giove}, free \cite{bessis} and rank 3 \cite{rank3} cases. This was instrumental in proving the long-standing $K(\pi,1)$ conjecture for Artin groups of affine type \cite{giove} and of rank 3 \cite{rank3}.

\subsection{Main results}
In this paper, we define a surjective group morphism $\Psi:\Art(W,S)\rightarrow \Art^*(W,T,h)$ (Proposition \ref{prop-def-psi}), generalizing a construction from the spherical \cite{bessis2002dual} and affine \cite{mccammond2017artin} cases. We then give conditions that are sufficient, necessary or equivalent to $\Psi$ being an isomorphism. Most of those conditions are related to the Hurwitz action, which is an action of the braid group on tuples of elements of a group, and in particular is well-defined on the reduced $T$-words of an element of $W$. Of particular note is the following:

\vspace{8pt}
\noindent\textbf{Theorem \ref{thm-focus}.} Let $S=\{s_1,\dots,s_n\}$, such that $h=s_1\dots s_n$. Let $\overline{s}_1,\dots,\overline{s}_n$ be the generators of $\Art(W,S)$ corresponding to $s_1,\dots,s_n$. If the following conditions are met, then $\Psi$ is an isomorphism:
\begin{itemize}
\item $(W,(s_1,\dots,s_n))$ is \emph{well-stabilized}, i.e. the stabilizers of $(s_1,\dots,s_n)$ and $(\overline{s}_1,\dots,\overline{s}_n)$ in the Hurwitz action are equal,
\item $(W,T,h)$ is \emph{pan-transitive}, i.e. the Hurwitz action on the set of reduced $T$-words of $a$ is transitive for each $a\in [1,h]_T$.
\end{itemize}
\vspace{8pt}

There are several cases where one or both of these conditions are known (Remark \ref{remark-known-cases}).

Now, let $(W_1,S_1),(W_2,S_2)$ be two Coxeter systems with Coxeter elements $h_1=s_1\dots s_k$ and $h_2=s_{k+1}\dots s_n$, and let $(W,S)=(W_1*W_2,S_1\sqcup S_2)$ with Coxeter element $h_1h_2$. Our main results are the following:

\vspace{8pt}
\noindent\textbf{Theorem \ref{thm-fg-hh}.} If $(W_1,(s_1,\dots,s_k))$ and $(W_2,(s_{k+1},\dots,s_n))$ are well-stabilized, then so is $(W,(s_1,\dots,s_n))$.

\vspace{8pt}
\noindent\textbf{Theorem \ref{thm-ele}.} If $(W_1,T_1,h_1)$ and $(W_2,T_2,h_2)$ are pan-transitive, then so is $(W,T,h)$.
\vspace{8pt}

From these, we get that if the two conditions from Theorem \ref{thm-focus} are satisfied for both $(W_1,S_1,h_1)$ and $(W_2,S_2,h_2)$, then these are also satisfied for $(W,S,h)$, and thus $\Psi:\Art(W,S)\rightarrow \Art^*(W,T,h)$ is an isomorphism (Theorem \ref{thm-final}). It is worth noting, however, that not all Coxeter elements of $(W,S)$ can be written as the product of a Coxeter element of $(W_1,S_1)$ and a Coxeter element of $(W_2,S_2)$.

Finally, we notice that the final result also holds if we consider the direct product instead of the free product (Theorem \ref{thm-direct}).

\subsection{Techniques used}
The dual Artin group $\Art^*(W,T,h)$ is defined from the \emph{noncrossing partition poset} $[1,h]_T$, which is well-studied in the free case \cite{bessis}. In particular, we can see the free group $F_n$ as the fundamental group of $\mathbb C$ with $n$ punctures, and say that an element of $F_n$ is \emph{noncrossing} if it is represented by a \emph{noncrossing loop}, i.e. a loop that does not self-intersect and is followed counterclockwise. There is a partial order on the set of noncrossing loops, which is given by the inclusion of the ``interior" of the loops. This passes to a partial order on $\text{NC}$, the set of noncrossing elements of $F_n$, and $[1,h]_T$ is isomorphic to the subposet $\text{NC}_g$ of elements in $\text{NC}$ below a fixed maximal element $g$ (Theorems \ref{ncg-equal-1g} and \ref{thm-nc-dual-free}).

In the case of a general Coxeter system $(W,S)$, there is an obvious projection $F_n\twoheadrightarrow W$, which restricts to a surjective poset morphism between the noncrossing partition posets that has nice properties (Proposition \ref{poset-morph-free-cox}). Using this, we can relate the noncrossing partition poset of any Coxeter system with fixed Coxeter element to $\text{NC}_g$. This allows us to see a noncrossing loops $\gamma$ as representing an element $a\in [1,h]_T$. In the case in which $W$ is a free product $W=W_1*W_2$, we can see segments of $\gamma$ as representing the elements that appear in the unique reduced factorization of $a$ in elements of $W_1\cup W_2$ (Figure \ref{fig-ra2}).

Another key piece of the puzzle is the action $\star$ of the braid group $\Br_n$ on the free group $F_n$, where we see $\Br_n$ as the mapping class group of a disk with $n$ punctures relative to its border, and $F_n$ as the fundamental group of said punctured disk. We define this action in a specific way that relates it closely to the Hurwitz action (Corollary \ref{cor-starhur}). This allows us to state conditions that are equivalent to the pan-transitive condition in terms of the $\star$ action (Proposition \ref{prop-rainbow-tfae}). The simplest of this conditions states that, if $H$ is the stabilizer of $(s_1,\dots,s_n)$ with respect to the Hurwitz actions, then for each $a\in [1,h]_T$ the action of $H$ via $\star$ is transitive on the set of elements of $\text{NC}_g$ that are projected onto $a$. There are other equivalent conditions which deal with certain tuples of elements of $[1,h]_T$ instead of a single element, and those are key in the proof of Theorem \ref{thm-ele}.

\subsection{Structure of this paper}
In Section \ref{sect-bg}, we give the basic definitions about Coxeter and Artin groups, we define the noncrossing partition poset and the dual Artin group. We then define the Hurwitz action, and we recall the results on free groups from \cite{bessis} that we need.

In Section \ref{sect-morph}, we relate the noncrossing partition poset of a general Coxeter system to that of a free case, by proving that the obvious projection from the free group to any Coxeter group restricts to a surjective poset morphism between the noncrossing partition posets that has nice properties (Proposition \ref{poset-morph-free-cox}). We then define the morphism $\Psi$, give a presentation for the dual Artin group that uses fewer relations, and use that presentation to give conditions that are sufficient, necessary or equivalent to $\Psi$ being an isomorphism, among which is Theorem \ref{thm-focus}.

In Section \ref{sect-ele} we give many conditions equivalent to the pan-transitive condition. First, we give some that are nicer properties about the poset morphism from Proposition \ref{poset-morph-free-cox}. Then, we define the action $\star$ of the braid group $\Br_n$ on the free group $F_n$, we explain how it is related to the Hurwitz action, and we prove the equivalence between the pan-transitive condition and some transitivity conditions of the $\star$ action (Proposition \ref{prop-rainbow-tfae}).

Finally, in Section \ref{sect-fg} we prove our main results. We make use of Proposition \ref{poset-morph-free-cox} to prove Theorem \ref{thm-fg-hh}, and we make use of the $\star$ action to prove Theorem \ref{thm-ele}.

\subsection{Acknowledgements}
The author thanks his Ph.D. supervisor Mario Salvetti for his valuable advice.

The author also thanks Luis Paris for sharing his unpublished notes \cite{parisnotes}.

Finally, the author acknowledges the MIUR Excellence Department Project awarded to the Department of Mathematics, University of Pisa, CUP I57G22000700001.
%\newpage
\section{Background}
\label{sect-bg}
\subsection{Coxeter and Artin groups}
\label{cox-bg}
Let $S$ be a set, and let $m:S \times S \rightarrow \{1,2,3,4,\dots\} \cup \{\infty\}$ be a symmetric funtion such that $m(s,s') = 1$ if and only if $s=s'$. Let $W$ the group presented as follows:
\[W \coloneqq \left<S\, \Big |\,  (ss')^{m(s,s')}=1 \quad \forall s,s'\in S \text{ such that } m(s,s')\ne\infty \right>\]
We say that $W$ is a \emph{Coxeter group}, and that $S$ is its set of \emph{simple reflections}. The couple $(W,S)$ is called a \emph{(standard) Coxeter system}, the function $m$ is called its \emph{Coxeter matrix}. The \emph{Coxeter graph} of $(W,S)$ is the graph that has $S$ as the set of vertices, and has an edge between $s$ and $s'$ labeled with $m(s,s')$ for each $s,s'\in S$ such that $m(s,s')\ge 3$. The \emph{rank} of $(W,S)$ is the cardinality of $S$. In this paper, we assume that the rank is a finite number $n$. A \emph{Coxeter element} for $(W,S)$ is the product of all elements of $S$ in some order. The \emph{set of reflections} of $(W,S)$ is the set $T$ of all elements of $W$ that are conjugate to an element of $S$. The couple $(W,T)$ is called a \emph{dual Coxeter system}.

Let $S=\{s_1, \dots, s_n\}$. We can define the \emph{(standard) Artin group}, also known as the \emph{generalized braid group}, associated to $(W,S)$ as follows:
\[\Art(W,S) \coloneqq \left<\overline{s}_1,\overline{s}_2,\dots,\overline{s}_n\, \Big |\,  [\overline{s}_i|\overline{s}_j]_{m(s_i,s_j)} \! = \! [\overline{s}_j|\overline{s}_i]_{m(s_i,s_j)} \; \forall i<j \text{ such that } m(s_i,s_j)\ne\infty \right>\]
Where we denote by $[a|b]_k$ the product $ababa\dots$ with $k$ factors. Notice that if we added the relations $\overline{s}_i^2=1$  for each $i$ we would obtain another presentation for the Coxeter group.

\subsection{Reduced words and lengths}
Let $G$ be a group, let $a\in G$, and let $X$ be a generating subset of $G$. We say that $(g_1, \dots, g_k)$ is an \emph{$X$-word} of \emph{length} $k$ representing $a$ if $g_1 g_2 \dots g_k = a$ and for each $i$ we have $g_i \in X$ or $g_i^{-1} \in X$. We denote by $\len_X(a)$ the minimal length of an $X$-word representing $a$, and by $\Red_X(a)$ the set of \emph{reduced $X$-words} representing $a$, i.e. the set of $X$-words representing $a$ of length $\len_X(a)$. We define a partial order relation $\le_X$ on $G$ as follows: we have $a \le_X b$ if and only if there exists $c\in G$ such that $ac=b$ and $\len_X(a)+\len_X(c)=\len_X(b)$. Equivalently, $a\le_X b$ if and only if there exists a reduced $X$-word representing $b$ containing an $X$-word representing $a$ as a prefix. If $X$ is closed by conjugation, we also have $a\le_X b$ if and only if there exists a reduced $X$-word representing $b$ containing an $X$-word representing $a$ as a suffix. In any case, if $a<_X b$, we denote the interval between $a$ and $b$ in the poset $(G,\le_X)$ by $[a,b]_X$.

\subsection{The dual Artin group}
\label{artinduale}

Let $(W,S)$ be a standard Coxeter system with set of reflections $T$. We fix a Coxeter element $h$. The interval $[1,h]_T$ is also known as the \emph{noncrossing partition poset} associated to $(W,T,h)$. We endow $[1,h]_T$ with an edge-labeling, giving a label associated with $t\in T\cap [1,h]_T$, which we denote by $\{t\}$, to all intervals of the form $[a,at]_T$ such that $1 \le_T a <_T at \le_T h$.

The \emph{dual Artin group}, also known as the \emph{dual braid group}, associated with $(W,T,h)$ is the group $\Art^*(W,T,h)$ generated by the labels of $[1,h]_T$, with the following relations:
\[\{t_1\}\{t_2\}\dots\{t_k\}=\{t_1'\}\{t_2'\}\dots\{t_k'\}\quad\forall a\in [1,h]_T\;\forall (t_1,\dots,t_k),(t_1',\dots,t_k')\in\Red_T(a)\]
For each $a\in [1,h]_T$, we denote by $\{a\}$ the element of $\Art^*(W,T,h)$ equal to the product $\{t_1\}\{t_2\}\dots\{t_k\}$, where $(t_1,\dots,t_k)\in\Red_T(a)$.

It is conjectured that the stanrdard Artin group $\Art(W,S)$ and the dual Artin group $\Art^*(W,T,h)$ are always isomorphic, for each Coxeter system $(W,S)$ and for each Coxeter element $h$.

\subsection{The Hurwitz action}
Let $k$ be a positive integer. The \emph{braid group on $k$ elements} is defined as follows:
\[\Br_k \coloneqq \left<\sigma_1, \dots,\sigma_{k-1}\,|\,\sigma_i\sigma_{i+1}\sigma_i=\sigma_{i+1}\sigma_i\sigma_{i+1}\quad\forall i\in\{1,\dots,k-2\}\right>\]
Given a group $G$, there is an action of $\Br_k$ on $G^k$, called the \emph{Hurwitz action}, defined as follows:
\[\begin{split}\sigma_i \cdot (g_1, \dots, g_k) = &\, (g_1, \dots, g_{i-1}, g_i g_{i+1} g_i^{-1}, g_i, g_{i+2}, \dots, g_k) \\
\sigma_i^{-1} \cdot (g_1, \dots, g_k) = &\, (g_1, \dots, g_{i-1},g_{i+1},g_{i+1}^{-1} g_i g_{i+1}, g_{i+2}, \dots, g_k)\end{split}\]
Notice that, if $X$ is a generating subset of $G$ that is closed by conjugation and $a\in G$, then the Hurwitz action of $\Br_{\len_X(a)}$ on $\Red_X(a)$ is well-defined. In particular, if $(W,T)$ is a rank $n$ dual Coxeter system with Coxeter element $h$, then the Hurwitz action of $\Br_n$ on $\Red_T (h)$ is well-defined.
\begin{teo}
The Hurwitz action on $\Red_T(h)$ is transitive. (\cite{igusa2009exceptional}, see also \cite{transhurwitz})
\label{hurwitz-cox}
\end{teo}

\subsection{The case of free groups}
\label{subs-fg}
All results in this Subsection are from \cite{bessis}.

Let $F_n$ be the free group freely generated by $f_1, \dots, f_n$, and let $g\coloneqq f_1\dots f_n$. We choose pairwise distinct points $x_0,x_1,\dots,x_n\in\C$ and identify $F_n$ with $\pi_1(\C\setminus\{x_1,\dots,x_n\},x_0)$, in such a way that $f_1,\dots,f_n,g$ are represented by the loops drawn in Figure \ref{fig-f1fn}.
\begin{figure}[h]
\centering

\begin{tikzpicture}
\draw [thick] (0,0) circle [radius=3];
\draw [fill] (-3,0) circle [radius=0.05];
\node[left] at (-3,0) {$x_0$};
\node[right] at (3,0) {$g$};
\node at (0,0.95) {$\vdots$};
\draw [fill] (0,-2.1) circle [radius=0.05];
\node[left] at (0,-2.1) {$x_1$};
\node[right] at (0.55,-2.1) {$f_1$};
\draw [fill] (0,-1.15) circle [radius=0.05];
\node[left] at (0,-1.15) {$x_2$};
\node[right] at (0.55,-1.15) {$f_2$};
\draw [fill] (0,-0.2) circle [radius=0.05];
\node[left] at (0,-0.2) {$x_3$};
\node[right] at (0.55,-0.2) {$f_3$};
\draw [fill] (0,2.1) circle [radius=0.05];
\node[left] at (0,2.1) {$x_n$};
\node[right] at (0.55,2.1) {$f_n$};
\draw [thick] (-3,0) to [out=-45,in=180] (0,-1.75) -- (0.2,-1.75) arc [radius=0.35, start angle=90, end angle=-90] -- (0,-2.45) to [out=180,in=-60] (-3,0);
\draw [thick] (-3,0) to [out=-20,in=180] (0,-0.8) -- (0.2,-0.8) arc [radius=0.35, start angle=90, end angle=-90] -- (0,-1.5) to [out=180,in=-35] (-3,0);
\draw [thick] (-3,0) to [out=5,in=180] (0,0.15) -- (0.2,0.15) arc [radius=0.35, start angle=90, end angle=-90] -- (0,-0.55) to [out=180,in=-10] (-3,0);
\draw [thick] (-3,0) to [out=60,in=180] (0,2.45) -- (0.2,2.45) arc [radius=0.35, start angle=90, end angle=-90] -- (0,1.75) to [out=180,in=45] (-3,0);
\end{tikzpicture}
\caption{Loops representing $f_1,f_2,\dots,f_n,g$. The loops are followed counterclockwise.}
\label{fig-f1fn}
\end{figure}
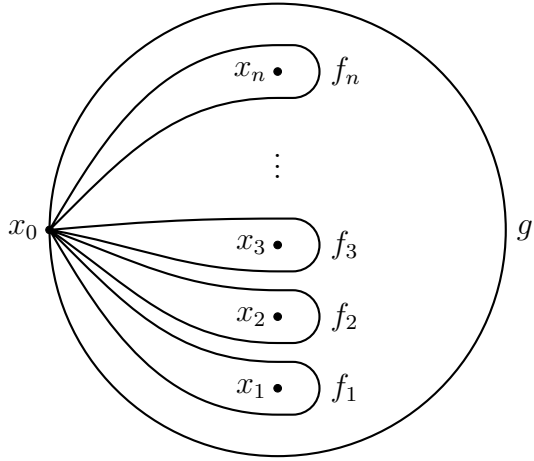

We say that a loop in $\C\setminus\{x_1,\dots,x_n\}$ based in $x_0$ is \emph{noncrossing} if it does not self-intersect and it is followed counterclockwise. Each noncrossing loop $\gamma$ divides $\C$ in two connected components, and we denote by $\Int(\gamma)$ the bounded one. We endow the set of noncrossing loops with a partial order relation $\subseteq$, such that $\gamma\subseteq\delta$ if and only if $\Int(\gamma)\subseteq\Int(\delta)$.

We say that an element of $F_n$ is \emph{noncrossing} if it is represented by a noncrossing loop. We denote by $\NC$ the set of noncrossing elements of $F_n$, and we endow $\NC$ with the partial order relation $\subseteq$, such that $a\subseteq b$ if and only if there exist noncrossing loops $\alpha,\beta$ representing $a,b$ respectively such that $\alpha\subseteq\beta$. We denote the interval $[1,g]$ in the poset $(\NC,\subseteq)$ by $\NC_g$.

Let $R\subseteq F_n$ be the set of all elements that are conjugate to an element of $\{f_1, \dots, f_n\}$. In this Subsection, let $S=\{s_1,\dots,s_n\}$, $(W,S)$ the Coxeter system such that $s_i s_j$ has infinite order for each $i\ne j$, let $h\coloneqq s_1\dots s_n$, and let $T$ the set of reflections of $(W,S)$.
\begin{teo}
For each $a\in[1,g]_R$ and for each $b\in[1,h]_T$, the Hurwitz actions of $\Br_n$ on $\Red_R(a)$ and $\Red_T(b)$ are both free and transitive.
\label{hurwitz-fg}
\end{teo}
\begin{teo}
\label{ncg-equal-1g}
The poset $(\NC_g,\subseteq)$ is a lattice and is equal to $([1,g]_R,\le_R)$. 
\end{teo}
\begin{teo}
The group morphism $F_n\rightarrow W$ given by $f_i\mapsto s_i$ restricts to a poset isomorphism $([1,g]_R,\le_R)\rightarrow ([1,h]_T,\le_T)$.
\label{thm-nc-dual-free}
\end{teo}

\section{A morphism from the standard Artin group to the dual Artin group}
\label{sect-morph}
Throughout Section \ref{sect-morph}, we fix a Coxeter system $(W,S)$, we fix $T$ its set of reflections, we fix an $n$-tuple $(s_1,\dots,s_n)$ such that $S=\{s_1,\dots,s_n\}$, and we fix the Coxeter element $h=s_1\dots s_n$. We shall also use the notation for the elements of $\Art^*(W,T,h)$ introduced in Subsection \ref{artinduale}. Moreover, we define $F_n,f_1,\dots,f_n,g,R$ as in Subsection \ref{subs-fg}, and we define $\pi:F_n\rightarrow W$ as the group morphism given by $f_i\mapsto s_i$.

\subsection{$[1,g]_R$ projects onto $[1,h]_T$}
We start by proving a generalization of Theorem \ref{thm-nc-dual-free}.
\begin{prop}
\label{poset-morph-free-cox}
The group morphism $\pi:F_n\rightarrow W$ restricts to a surjective poset morphism $([1,g]_R,\le_R)\rightarrow([1,h]_T,\le_T)$ that preserves height. Moreover, given $x,y\in [1,h]_T$, we have $x\le_T y$ if and only if there exist $\tilde{x},\tilde{y}\in [1,g]_R$ such that $\tilde{x}\le_R \tilde{y}$, $\pi(\tilde{x})=x$ and $\pi(\tilde{y})=y$.
\end{prop}
\begin{proof}
First, we shall prove $\pi([1,g]_R)\subseteq [1,h]_T$, and that $\pi$ preserves height. Let $a\in [1,g]_R$, let $k\coloneqq\len_R(a)$ and let $c=a^{-1}g$. Since $a\le_R g$, we know $\len_R(c)=n-k$. Fix $(r_1,\dots,r_k)\in\Red_R(a)$ and $(r_{k+1},\dots,r_n)\in\Red_R(c)$. We then have $(r_1,\dots,r_n)\in\Red_R(g)$. By Theorem \ref{hurwitz-fg}, there exists $\tau\in\Br_n$ such that $(r_1,\dots,r_n)=\tau\cdot (f_1,\dots,f_n)$. We then define $(t_1,\dots,t_n)=\tau\cdot (s_1,\dots,s_n)$, and notice that $\pi(a)=t_1\dots t_k$ is an element of $[1,h]_T$ of height $k$.

We shall now prove that $\pi$ maps $[1,g]_R$ surjectively onto $[1,h]_T$. Let $a\in [1,h]_T$, let $k\coloneqq\len_T(a)$ and let $c=a^{-1}h$. Since $a\le_T h$, we know $\len_T(c)=n-k$. Fix $(t_1,\dots,t_k)\in\Red_T(a)$ and $(t_{k+1},\dots,t_n)\in\Red_T(c)$. We then have $(t_1,\dots,t_n)\in\Red_T(h)$. By Theorem \ref{hurwitz-cox}, there exists $\tau\in\Br_n$ such that $(t_1,\dots,t_n)=\tau\cdot (s_1,\dots,s_n)$. We then define $(r_1,\dots,r_n)=\tau\cdot (f_1,\dots,f_n)$, and notice that $\tilde{a}\coloneqq r_1\dots r_k\in [1,g]_R$ is such that $\pi(\tilde{a})=a$.

We shall now prove that $\pi$ restricts to a poset morphism $([1,g]_R,\le_R)\rightarrow([1,h]_T,\le_T)$. Let $\tilde{x},\tilde{y}\in [1,g]_R$ such that $\tilde{x}\le_R \tilde{y}$, and let $x\coloneqq\pi(\tilde{x})$ and $y\coloneqq\pi(\tilde{y})$. Let $k\coloneqq\len_R(\tilde{x})$ and $\ell\coloneqq\len_R(\tilde{y})$. Since $\tilde{x}\le_R\tilde{y}\le_R g$, we know $\len_R(\tilde{x}^{-1}\tilde{y})=\ell-k$ and $\len_R(\tilde{y}^{-1}g)=n-\ell$. Fix $(r_1,\dots,r_k)\in\Red_R(\tilde{x})$, $(r_{k+1},\dots,r_\ell)\in\Red_R(\tilde{x}^{-1}\tilde{y})$ and $(r_{\ell+1},\dots,r_n)\in\Red_R(\tilde{y}^{-1}g)$. We then have $(r_1,\dots,r_n)\in\Red_R(g)$. By Theorem \ref{hurwitz-fg}, there exists $\tau\in\Br_n$ such that $(r_1,\dots,r_n)=\tau\cdot (f_1,\dots,f_n)$. We then define $(t_1,\dots,t_n)=\tau\cdot (s_1,\dots,s_n)$, and notice that $x=t_1\dots t_k\le_T t_1\dots t_\ell=y$.

Finally, we shall fix $x,y\in [1,h]_T$ such that $x\le_T y$ and prove that there exist $\tilde{x},\tilde{y}\in [1,g]_R$ such that $\tilde{x}\le_R \tilde{y}$, $\pi(\tilde{x})=x$ and $\pi(\tilde{y})=y$. Let $k\coloneqq\len_T(x)$ and $\ell\coloneqq\len_T(y)$. Since $x\le_T y\le_T h$, we know $\len_R(x^{-1}y)=\ell-k$ and $\len_R(y^{-1}h)=n-\ell$. Fix $(t_1,\dots,t_k)\in\Red_T(x)$, $(t_{k+1},\dots,t_\ell)\in\Red_T(x^{-1}y)$ and $(t_{\ell+1},\dots,t_n)\in\Red_T(y^{-1}h)$. We then have $(t_1,\dots,t_n)\in\Red_T(h)$. By Theorem \ref{hurwitz-cox}, there exists $\tau\in\Br_n$ such that $(t_1,\dots,t_n)=\tau\cdot (s_1,\dots,s_n)$. We then define $(r_1,\dots,r_n)=\tau\cdot (f_1,\dots,f_n)$, $\tilde{x}\coloneqq r_1\dots r_k$ and $\tilde{y}=r_1\dots r_\ell$ and notice that they have the properties we desire.
\end{proof}

\subsection{The Hurwitz presentation for $\Art^*(W,T,h)$}
\label{sub-fr}
From the presentation of $\Art^*(W,T,h)$ given in Subsection \ref{artinduale}, we recall that $\Art^*(W,T,h)$ is generated by the elements of the form $\{t\}$, where $t\in T \cap [1,h]_T$. In this Subsection, we shall give another presentation for $\Art^*(W,T,h)$ that uses these generators and fewer relations. This is the presentation called \emph{Hurwitz presentation} in \cite{mccpres}. We will provide our own proof as we will later need some of its details for the proof of Proposition \ref{prop-isom-iff}.

First, we define the equivalence relation $\approx$ on $\Br_n$, such that $\tau\approx\tau'$ if and only if the last element of the $n$-tuple $\tau\cdot (s_1,\dots,s_n)$ is equal to the last element of the $n$-tuple $\tau'\cdot (s_1,\dots,s_n)$. We thus have a bijection $\varphi:\Br_n/\!\approx\;\rightarrow T\cap[1,h]_T$, such that $\varphi([\tau])$ is the last element of $\tau\cdot (s_1,\dots,s_n)$.

Then, for each class $c\in\Br_n/\!\approx$, we define a symbol $a_c$, and we define $\frak{F}$ as the free group freely generated by those symbols. For each $\tau \in \Br_n$ and for each $0 < i \le j < n$, we define the following elements of $\frak{F}$:
\[p_{\tau,j}\coloneqq a_{[\sigma_{n-1} \dots \sigma_{n-j} \tau]} \dots a_{[\sigma_{n-1} \sigma_{n-2} \tau]} a_{[\sigma_{n-1} \tau]}  a_{[\tau]}\]
\[r_{\tau,i,j}\coloneqq p_{\tau,j} p_{\sigma_{n-i} \tau,j}^{-1} \]
Finally, we define $\frak{R}$ as the smallest normal subgroup of $\frak{F}$ containing all elements of the form $r_{\tau,i,j}$.

\begin{prop}
Let $\xi:\frak{F}\rightarrow \Art^*(W,T,h)$ be the group morphism given by $a_c\mapsto{\varphi(c)}$. Then, $\xi$ is a surjective morphism, and $\Ker(\xi)=\frak{R}$. Therefore, $\Art^*(W,T,h)$ is isomorphic to $\frak{F}/\frak{R}$.
\label{prop-anotherpres}
\end{prop}
\begin{proof}
We have already noted that $\Art^*(W,T,h)$ is generated by the labels of the form $\{t\}$, where $t\in T \cap [1,h]_T$, and that $\varphi:\Br_n/\!\approx\;\rightarrow T\cap[1,h]_T$ is a bijection. This proves that $\xi$ is surjective. Assuming $\Ker(\xi)=\frak{R}$, the last sentence of the statement follows from the first isomorphism theorem.

We shall now prove $\Ker(\xi)\supseteq\frak{R}$. Let $\tau \in \Br_n$, $0 < i \le j < n$ and $(t_1, \dots, t_n) = \tau \cdot (s_1, \dots, s_n)$. We then have $\xi(p_{\tau, j}) =  \{t_{n-j}\}\{t_{n-j+1}\}\dots\{t_{n}\}$. From the definitions given in Subsection \ref{artinduale}, this yields $\xi(p_{\tau, j}) =  \{t_{n-j}t_{n-j+1}\dots t_{n}\}$. In the same way, one can prove:
\[\xi(p_{\sigma_{n-i}\tau, j})=\{t_{n-j}\dots t_{n-i-1}\cdot t_{n-i}t_{n-i+1}t_{n-i}^{-1}\cdot t_{n-i}\cdot t_{n-i+2}\dots t_n\}\]
Which yields $\xi(p_{\tau, j})=\xi(p_{\sigma_{n-i}\tau, j})$, and thus $\xi(r_{\tau,i,j})=1$. Given the definition of $\frak{R}$, this proves $\Ker(\xi)\supseteq\frak{R}$.

We shall now prove $\Ker(\xi)\subseteq\frak{R}$. Given the presentation of $\Art^*(W,T,h)$ from Subsection \ref{artinduale}, it is enough to fix $a\in [1,h]_T$, $j\coloneqq \len_T (a)-1$, and $(t_j,\dots,t_n),(t_j',\dots,t_n')\in\Red_T(a)$ and prove the following:
\begin{equation}
\left(a_{\varphi^{-1}(t_{n-j})}a_{\varphi^{-1}(t_{n-j+1})}\dots a_{\varphi^{-1}(t_n)}\right)\left(a_{\varphi^{-1}(t_{n-j}')}a_{\varphi^{-1}(t_{n-j+1}')}\dots a_{\varphi^{-1}(t_n')}\right)^{-1}\in\frak{R}
\label{eqn-altpres}
\end{equation}
Fix $(t_1,\dots,t_{n-j-1})\in\Red_T(ha^{-1})$. Using Theorem \ref{hurwitz-cox}, we fix $\alpha,\tau\in\Br_n$ such that:
\[\alpha\cdot (s_1,\dots,s_n)=(t_1,\dots,t_{n-j-1},t_{n-j},\dots,t_n)\]
\[\tau\cdot (t_1,\dots,t_{n-j-1},t_{n-j},\dots,t_n)=(t_1,\dots,t_{n-j-1},t_{n-j}',\dots,t_n')\]
So \eqref{eqn-altpres} is equivalent to $p_{\alpha,j}p_{\tau\alpha,j}^{-1}\in\frak{R}$. We also have:
\[p_{\alpha,n-1}p_{\tau\alpha,n-1}^{-1}=\left(a_{\varphi^{-1}(t_{1})}a_{\varphi^{-1}(t_{2})}\dots a_{\varphi^{-1}(t_{n-j-1})}\right)p_{\alpha,j}p_{\tau\alpha,j}^{-1}\left(a_{\varphi^{-1}(t_{1})}a_{\varphi^{-1}(t_{2})}\dots a_{\varphi^{-1}(t_{n-j-1})}\right)^{-1}\]
Therefore, since $\frak{R}$ is a normal subgroup of $\frak{F}$, \eqref{eqn-altpres} is equivalent to $p_{\alpha,n-1}p_{\tau\alpha,n-1}^{-1}\in\frak{R}$.

Let $(\beta_1,\dots,\beta_k)$ be a $\{\sigma_1,\dots,\sigma_{n-1}\}$-word for $\tau$, and for each $i\in\{0,\dots,k\}$ let $\tau_i\coloneqq\beta_{k-i+1}\beta_{k-i+2}\dots\beta_k$. We have:
\[\begin{split}
p_{\alpha,n-1}p_{\tau\alpha,n-1}^{-1}=&p_{\tau_0\alpha,n-1}p_{\tau_k\alpha,n-1}^{-1}\\
=&\left(p_{\tau_0\alpha,n-1}p_{\tau_1\alpha,n-1}^{-1}\right)\left(p_{\tau_1\alpha,n-1}p_{\tau_2\alpha,n-1}^{-1}\right)\dots\left(p_{\tau_{k-1}\alpha,n-1}p_{\tau_k\alpha,n-1}^{-1}\right)
\end{split}\]
Fix $\ell\in\{0,\dots,k-1\}$. We have that $\beta_{n-\ell}$ is either equal to $\sigma_{n-i}$ or to $\sigma_{n-i}^{-1}$ for some $i$. In the former case, we have $p_{\tau_{k-\ell-1}\alpha,n-1}p_{\tau_{k-\ell}\alpha,n-1}^{-1}=r_{\tau_{k-i}\alpha,i,n-1}^{-1}$. In the latter, we have $p_{\tau_{k-\ell-1}\alpha,n-1}p_{\tau_{k-\ell}\alpha,n-1}^{-1}=r_{\tau_{k-i-1}\alpha,i,n-1}$. Either way, we have $p_{\tau_{k-\ell-1}\alpha,n-1}p_{\tau_{k-\ell}\alpha,n-1}^{-1}\in\frak{R}$. This proves $p_{\alpha,n-1}p_{\tau\alpha,n-1}^{-1}\in\frak{R}$, and thus the statement.
\end{proof}
For the remainder of Section \ref{sect-morph}, we shall identify $\Art^*(W,T,h)$ with $\frak{F}/\frak{R}$ via the isomorphism provided by Proposition \ref{prop-anotherpres}. With this identification, we shall now re-write the relations $r_{\tau,i,j}$ in a simpler way, obtaining the Hurwitz presentation.

\begin{cor}
The group $\Art^*(W,T,h)$ is generated by the labels of $[1,h]_T$, with the following relations:
\[\{t\}\{t'\}=\{t\, t' t\}\{t\}\quad\forall t,t'\in T\text{ such that }t\, t'\in [1,h]_T\]
\end{cor}
\begin{proof}
Fix $\tau\in\Br_n$ and $0 < i \le j < n$, and let $(t_1,\dots, t_n)\coloneqq \tau\cdot (s_1,\dots, s_n)$. We then have:
\[\frak{F}/\frak{R}\ni [p_{\tau,j}]=\{t_{n-j}\}\{t_{n-j+1}\}\dots\{t_{n}\}\in\Art^*(W,T,h)\]
\[\frak{F}/\frak{R}\ni [p_{\sigma_{n-i}\tau,j}]=\{t_{n-j}\}\dots\{t_{n-i-1}\}\{t_{n-i}t_{n-i+1}t_{n-i}\}\{t_{n-i}\}\{t_{n-i+2}\}\dots\{t_{n}\}\in\Art^*(W,T,h)\]
So the relation $r_{\tau,i,j}$ can be rewritten as:
\[\begin{split}
\{t_{n-j}\}\{&t_{n-j+1}\}\dots\{t_{n}\}\\
&=\{t_{n-j}\}\dots\{t_{n-i-1}\}\{t_{n-i}t_{n-i+1}t_{n-i}\}\{t_{n-i}\}\{t_{n-i+2}\}\dots\{t_{n}\}
\end{split}\]
Which, multiplying both sides by $(\{t_{n-j}\}\dots\{t_{n-i-1}\})^{-1}$ on the left and $(\{t_{n-i+2}\}\dots\{t_{n}\})^{-1}$ on the right, is equivalent to:
\[\{t_{n-i}\}\{t_{n-i+1}\}=\{t_{n-i}t_{n-i+1}t_{n-i}\}\{t_{n-i}\}\]
From this the statement follows easily.
\end{proof}

\subsection{Definition of the morphism $\Psi$}
In this Subsection, we shall define a surjective group morphism $\Psi$ from $\Art(W,S)$ to $\Art^*(W,T,h)$.
This morphism was already defined in the spherical \cite{bessis2002dual} and affine \cite{mccammond2017artin} cases, and defined independently in the general case by Paris \cite{parisnotes}.

Before defining that morphism, we shall notice that the Hurwitz presentation of $\Art^*(W,T,h)$ yields the following lemma:
\begin{lemma}
Let $a\in [1,h]_T$, $k\coloneqq\len_T(a)$, $(t_1,\dots,t_k)\in\Red_T(a)$, $\tau\in\Br_k$ and $(t_1',\dots,t_k')\coloneqq\tau\cdot (t_1,\dots,t_k)$. Then, using the notation from Subsection \ref{artinduale}, we have $(\{t_1'\},\dots,\{t_k'\})\coloneqq\tau\cdot (\{t_1\},\dots,\{t_k\})$.
\label{hurwitz-duale}
\end{lemma}

Which itself yields:

\begin{cor}
$\Art^*(W,T,h)$ is generated by $\{s_1\},\{s_2\},\dots,\{s_n\}$.
\label{dual-gen-s}
\end{cor}
\begin{proof}
From Subsection \ref{artinduale}, we know that $\Art^*(W,T,h)$ is generated by elements of the form $\{t\}$, where $t\in T \cap [1,h]_T$. Fix one such $t$, and fix $(t_1,\dots,t_{n-1})\in\Red_T(ht^{-1})$. By Theorem \ref{hurwitz-cox}, there exists $\tau\in\Br_n$ such that $\tau\cdot (s_1,\dots,s_n)=(t_1,\dots,t_{n-1},t)$. From Lemma \ref{hurwitz-duale} we then know that $\{t\}$ is the last element of $\tau\cdot (\{s_1\},\dots,\{s_n\})$, and thus belongs to the group generated by $\{s_1\},\{s_2\},\dots,\{s_n\}$. This proves the statement.
\end{proof}

We now get to the definition of $\Psi$:

\begin{prop}
\label{prop-def-psi}
The group morphism $\Psi:\Art(W,S)\rightarrow \Art^*(W,T,h)$ given by $\overline{s}_i\mapsto\{s_i\}$ is well-defined and surjective.
\end{prop}
\begin{proof}
Recalling the presentation for $\Art(W,S)$ given in Subsection \ref{cox-bg}, to prove that $\Psi$ is well-defined it is enough to fix $i<j$ such that the order $m$ of $s_is_j$ is finite and prove the following equality between elements of $\Art^*(W,T,h)$:
\begin{equation}
\left[\{s_i\}|\{s_j\}\right]_m=\left[\{s_j\}|\{s_i\}\right]_m
\label{eqn-Psi-1}
\end{equation}
For each $z\in\Z$, let $x_z$ be the second element of the couple $\sigma_1^z\cdot (s_i,s_j)$. Therefore, we have $x_0=x_m=s_j$, $x_1=s_i$, and for each $z\in\Z$ we have $x_{z+1}x_z=s_is_j$. Assuming that $m$ is even, we rewrite \eqref{eqn-Psi-1} as follows:
\[(\{s_i\}\{s_j\})^{\frac m2}=\{s_j\}(\{s_i\}\{s_j\})^{\frac {m-2}2}\{s_i\}\]
\[\{s_is_j\}^{\frac m2}=\{s_j\}\{s_is_j\}^{\frac {m-2}2}\{s_i\}\]
\[\begin{split}
\{x_m\}\{x_{m-1}\}\cdot\{x_{m-2}&\}\{x_{m-3}\}\cdot\ldots\cdot\{x_2\}\{x_1\}=\\
&\{x_m\}\cdot\{x_{m-1}\}\{x_{m-2}\}\cdot\{x_{m-3}\}\{x_{m-4}\}\cdot\ldots\cdot\{x_3\}\{x_2\}\cdot\{x_1\}
\end{split}\]
With the last equality being evidently true, proving \eqref{eqn-Psi-1}. Assuming that $m$ is odd, we rewrite \eqref{eqn-Psi-1} as follows:
\[(\{s_i\}\{s_j\})^{\frac {m-1}2}\{s_i\}=\{s_j\}(\{s_i\}\{s_j\})^{\frac {m-1}2}\]
\[\{s_is_j\}^{\frac {m-1}2}\{s_i\}=\{s_j\}\{s_is_j\}^{\frac {m-1}2}\]
\[\begin{split}
\{x_m\}\{x_{m-1}\}\cdot\{x_{m-2}&\}\{x_{m-3}\}\cdot\ldots\cdot\{x_3\}\{x_2\}\cdot\{x_1\}=\\
&\{x_m\}\cdot\{x_{m-1}\}\{x_{m-2}\}\cdot\{x_{m-3}\}\{x_{m-4}\}\cdot\ldots\cdot\{x_2\}\{x_1\}
\end{split}\]
With the last equality being evidently true, proving \eqref{eqn-Psi-1}. This proves that $\Psi$ is a well-defined morphism. It is surjective by Corollary \ref{dual-gen-s}.
\end{proof}

\subsection{Conditions related to $\Psi$ being an isomorphism}
In this Subsection we shall give conditions that are sufficient, necessary or equivalent to the morphism $\Psi$ from Proposition \ref{prop-def-psi} being an isomorphism.

We start by defining an equivalence relation on $\Br_n$, similar to the relation $\approx$ from Subsection \ref{sub-fr}. We say that $\tau\dapprox\tau'$ if and only if the last element of the $n$-tuple $\tau\cdot (\overline{s}_1,\dots,\overline{s}_n)$ is equal to the last element of the $n$-tuple $\tau'\cdot (\overline{s}_1,\dots,\overline{s}_n)$, where $\overline{s}_1,\dots,\overline{s}_n$ are the generators of $\Art(W,S)$ from Subsection \ref{cox-bg}. Notice that, since $\Art(W,S)$ projects onto $W$, we have that if $\tau,\tau'\in\Br_n$ satisfy $\tau\dapprox\tau'$ then they also satisfy $\tau\approx\tau'$.

\begin{prop}
$\Psi$ is an isomorphism if and only if the relations $\approx$ and $\dot\approx$ are equal.
\label{prop-isom-iff}
\end{prop}
\begin{proof}
Assume that $\Psi$ is an isomorphism, and fix $\tau,\tau'\in\Br_n$ such that $\tau\approx\tau'$. Let $a,b$ be the last elements of the $n$-tuples $\tau\cdot (\overline{s}_1,\dots,\overline{s}_n)$, $\tau'\cdot (\overline{s}_1,\dots,\overline{s}_n)$ respectively. From $\tau\approx\tau'$ we have $\Psi(a)=\Psi(b)$, and since $\Psi$ is injective we have $a=b$, which proves $\tau\dapprox\tau'$. This proves that, if $\Psi$ is an isomorphism, then the relations $\approx$ and $\dot\approx$ are equal.

Assume that the relations $\approx$ and $\dot\approx$ are equal. This means that there is a well-defined group morphism $\Phi:\frak{F}\rightarrow \Art(W,S)$ given by $a_{[\tau]}\mapsto\overline{t}$, where $\overline{t}$ is the last element of the $n$-tuple $\tau\cdot (\overline{s}_1,\dots,\overline{s}_n)$. We shall now show $\frak{R}\subseteq\Ker(\Phi)$, from which we know that $\Phi$ induces a morphism $\overline{\Phi}:\frak{F}/\frak{R}\rightarrow \Art(W,S)$.

Fix $\tau \in \Br_n$, $0 < i \le j < n$, and $(\overline{t}_1, \dots, \overline{t}_n) \coloneqq \tau \cdot (\overline{s}_1, \dots, \overline{s}_n)$. We have:
\[\begin{split}
\Phi(r_{\tau,i,j}) = &\, \Phi(p_{\tau,j} p_{\sigma_{n-i} \tau,j}^{-1}) \\ 
= &\, \Phi\!\left(a_{[\sigma_{n-1} \dots \sigma_{n-j} \tau]} \dots  a_{[\sigma_{n-1} \tau]} a_{[\tau]} \cdot
a_{[\sigma_{n-i} \tau]}^{-1} a_{[\sigma_{n-1} \sigma_{n-i} \tau]}^{-1} \dots a_{[\sigma_{n-1} \dots \sigma_{n-j} \sigma_{n-i} \tau]}^{-1}
\right)\\
= &\, \overline{t}_{n-j} \dots \overline{t}_n \cdot \overline{t}_n^{-1} \dots \overline{t}_{n-i+2}^{-1} \overline{t}_{n-i}^{-1} \left(\overline{t}_{n-i}\overline{t}_{n-i+1}\overline{t}_{n-i}^{-1}\right)^{-1} \overline{t}_{n-i-1}^{-1} \dots \overline{t}_{n-j}^{-1} = 1
\end{split}\]
This proves $r_{\tau,i,j}\in\Ker(\Phi)$, and therefore $\frak{R}\subseteq\Ker(\Phi)$. We shall now show that $\Psi$ is an isomorphism. For each $i \in \{ 1, \dots, n\}$, we have:
\[\overline{\Phi}(\Psi(\overline{s}_i))=\overline{\Phi}([a_{\phi^{-1}(s_i)}])=\Phi(a_{[\sigma_{n-1} \dots \sigma_i]})=\overline{s}_i\]
This proves that $\overline{\Phi} \circ \Psi=\id_{\Art(W,S)}$, so $\Psi$ is injective. We already knew from Proposition \ref{prop-def-psi} that $\Psi$ is surjective, thus it is an isomorphism.
\end{proof}

We now define the subgroups $H(W,(s_1,\dots,s_n)),\overline{H}(W,(s_1,\dots,s_n))$ of $\Br_n$ as the stabilizers with respect to the Hurwitz action of $(s_1,\dots,s_n)$ and $(\overline{s}_1,\dots,\overline{s}_n)$ respectively. For the remainder of Section \ref{sect-morph}, since $W$ and $(s_1,\dots,s_n)$ will remain fixed, we shall denote those groups simply as $H$ and $\overline{H}$. Notice that, since $\Art(W,S)$ projects onto $W$, $\overline{H}$ is always a subgroup of $H$. We say that $(W,(s_1,\dots,s_n))$ is \emph{well-stabilized} if $H=\overline{H}$.

\begin{prop}
\label{isom-then-ws}
If $\Psi$ is an isomorphism, then $(W,(s_1,\dots,s_n))$ is well-stabilized.
\end{prop}
\begin{proof}
Assume by contradiction that $(W,(s_1,\dots,s_n))$ is not well-stabilized, and fix $\tau\in H\setminus\overline{H}$. Since $\tau\notin\overline{H}$, there exists $i\in\{1,\dots,n\}$ such that the $i$-th element of the $n$-tuple $\tau\cdot (\overline{s}_1,\dots,\overline{s}_n)$ is different from $\overline{s}_i$. Therefore, if we define $\beta\coloneqq\sigma_{n-1}\sigma_{n-2}\dots\sigma_i$, we have $\beta\approx\beta\tau$ but not $\beta\dapprox\beta\tau$. By Proposition \ref{prop-isom-iff}, we have that $\Psi$ is not an isomorphism, which contradicts the hypothesis.
\end{proof}

\begin{prop}
\label{stab-ele-isom}
Suppose that $(W,(s_1,\dots,s_n))$ is well-stabilized, and that for each $t\in T\cap[1,h]_T$ the Hurwitz action on $\Red_T(ht^{-1})$ is transitive. Then, $\Psi$ is an isomorphism.
\end{prop}
\begin{proof}
By Proposition \ref{prop-isom-iff}, it is enough to fix $\tau,\tau'\in\Br_n$ such that $\tau\approx\tau'$ and show $\tau\dapprox\tau'$.

Let $(t_1,\dots,t_n)\coloneqq\tau\cdot(s_1,\dots,s_n)$ and $(t_1',\dots,t_n')\coloneqq\tau'\cdot(s_1,\dots,s_n)$. From $\tau\approx\tau'$ we know $t_n=t_n'=:t$. Since the Hurwitz action on $\Red_T(ht^{-1})$ is transitive, there exists $\beta\in\Br_{n-1}$ such that $\beta\cdot (t_1,\dots,t_{n-1})=(t_1',\dots,t_{n-1}')$. Identifying $\Br_{n-1}$ with the subgroup of $\Br_n$ generated by $\sigma_1,\dots,\sigma_{n-2}$ via the inclusion $\sigma_i\mapsto\sigma_i$, we then have $\beta\tau\cdot(s_1,\dots,s_n)=\tau'\cdot(s_1,\dots,s_n)$, thus $(\tau')^{-1}\beta\tau\in H$. Since $H=\overline{H}$, we then have $\beta\tau\cdot(\overline{s}_1,\dots,\overline{s}_n)=\tau'\cdot(\overline{s}_1,\dots,\overline{s}_n)$. Since $\beta\in\Br_{n-1}$, this yields $\tau\dapprox\tau'$.
\end{proof}

We now state stronger conditions that imply the second part of the hypothesis from Proposition \ref{stab-ele-isom}.

\begin{prop}
\label{ele-forte}
Each of the following statements implies the next one:
\begin{enumerate}
\item For each $a\in [1,h]_T$, if $W'$ is the subgroup of $W$ generated by $T\cap [1,a]_T$, there exists $S'\subseteq T\cap [1,a]_T$ such that $(W',S')$ is a Coxeter system with Coxeter element $a$.
\item For each $a\in [1,h]_T$, the Hurwitz action on $\Red_T(a)$ is transitive.
\item For each $t\in T\cap[1,h]_T$ the Hurwitz action on $\Red_T(ht^{-1})$ is transitive.
\end{enumerate}
\end{prop}
\begin{proof}
The implication $1\Rightarrow 2$ follows from Theorem \ref{hurwitz-cox}, while $2\Rightarrow 3$ is immediate.
\end{proof}

If the second statement of Proposition \ref{ele-forte} holds, we say that $(W,T,h)$ is \emph{pan-transitive}. We now state the following theorem, which will be the focus of the rest of this paper:
\begin{teo}
\label{thm-focus}
If $(W,(s_1,\dots,s_n))$ is well-stabilized and $(W,T,h)$ is pan-transitive, then the groups $\Art(W,S)$ and $\Art^*(W,T,h)$ are isomorphic.
\end{teo}
\begin{proof}
Follows immediately from Propositions \ref{stab-ele-isom} and \ref{ele-forte}.
\end{proof}

\begin{oss}
\label{remark-known-cases}
In the following cases it is known that the first statement of Proposition \ref{ele-forte} holds:
\begin{itemize}
\item The spherical case \cite{bessis2002dual}, i.e. the case where $W$ is finite
\item The affine case \cite{giove}
\item The rank 3 case \cite{rank3}
\item The crystallographic case in the sense of Hubery-Krause \cite{huberykrause}, i.e. the case where the following hold:
\begin{itemize}
\item[(1)] $\forall s,s'\in S\quad m(s,s')\in\{2,3,4,6,\infty\}$, where $m$ is the Coxeter matrix of $(W,S)$
\item[(2)] In each circuit of the Coxeter graph not containing the edge label $\infty$, the number of edges labelled 4 is even, and so is the number of edges labelled 6.
\end{itemize}
\end{itemize}
From Proposition \ref{ele-forte}, we then know that in those cases $(W,T,h)$ is pan-transitive. Moreover, in the spherical \cite{bessis2002dual} and affine \cite{mccammond2017artin} cases, it is also known that $\Psi$ is an isomorphism, so by Proposition \ref{isom-then-ws} we know that in those cases $(W,(s_1,\dots,s_n))$ is well-stabilized. In \cite{rank3} it is proven that in the rank 3 case the groups $\Art(W,S)$, $\Art^*(W,T,h)$ are isomorphic, however it is not explicitly proven that $\Psi$ is an isomorphism.
\end{oss}

\section{The pan-transitive condition}
\label{sect-ele}
The aim of this Section is to give a deep dive on what the pan-transitive condition that appears in the statement of Theorem \ref{thm-focus} is, and to show many useful equivalent conditions.

Throughout Section \ref{sect-ele}, we shall continue to use all of the notations from Section \ref{sect-morph}. Moreover, if $a\in [1,h]_T$ and $\tilde{a}\in [1,g]_R\cap\pi^{-1}(a)$, we shall simply write $\tilde{a}\in a$. Finally, we shall identify $F_n$ with the fundamental group of $(\C\setminus\{x_1,\dots,x_n\},x_0)$ and define noncrossing loops, $\NC$ and $\NC_g$ as in Subsection \ref{subs-fg}, in order to make use of the equality between $(\NC_g,\subseteq)$ and $([1,g]_R,\le_R)$ from Theorem \ref{ncg-equal-1g}.

\subsection{Conditions related to $[1,g]_R$}
\begin{prop}
\label{ele-sse-1gt}
The following are equivalent:
\begin{enumerate}
\item For each $t,b\in [1,h]_T$ such that $t\le_T b$ and $t\in T$, and for each $\tilde{b}\in b$, there exists $\tilde{t}\in t$ such that $\tilde{t}\le_R\tilde{b}$.
\item For each $a,b\in [1,h]_T$ such that $a\le_T b$ and $a^{-1}b\in T$, and for each $\tilde{a}\in a$, there exists $\tilde{b}\in b$ such that $\tilde{a}\le_R\tilde{b}$.
\item For each $a,b\in [1,h]_T$ such that $a\le_T b$, and for each $\tilde{a}\in a$, there exists $\tilde{b}\in b$ such that $\tilde{a}\le_R\tilde{b}$.
\item For each $a,b\in [1,h]_T$ such that $a\le_T b$, and for each $\tilde{b}\in b$, there exists $\tilde{a}\in a$ such that $\tilde{a}\le_R\tilde{b}$.
\item $(W,T,h)$ is pan-transitive.
\end{enumerate}
\end{prop}
\begin{proof}
We shall prove the implications $1\Rightarrow 2$, $2\Rightarrow 3$, $3\Rightarrow 4$, $4\Rightarrow 5$ and $5\Rightarrow 1$.

\begin{enumerate}
\item[$1\Rightarrow 2$] Let $a,b,\tilde{a}$ as in statement 2, and let $t\coloneqq a^{-1}b$. Since $t\le_T a^{-1}h$ and $\tilde{a}^{-1}g\in a^{-1}h$, statement 1 yields the existence of $\tilde{t}\in t$ such that $\tilde{t}\le_R\tilde{a}^{-1}g$. We then define $\tilde{b}\coloneqq\tilde{a}\tilde{t}$.

\item[$2\Rightarrow 3$] Let $a,b,\tilde{a}$ as in statement 3, and let $(a_0,\dots,a_k)$ be a maximal chain in $[a,b]_R$. From $\tilde{a}\in a=a_0$ and statement 2, it is easy to prove by induction that for each $j\in\{1,\dots,k\}$ there exists $\tilde{a}_j\in a_j$ such that $\tilde{a}_j\ge_R\tilde{a}_{j-1}\ge_R\tilde{a}$. We then define $\tilde{b}\coloneqq\tilde{a}_k$.

\item[$3\Rightarrow 4$] Let $a,b,\tilde{b}$ as in statement 4. Since $b^{-1}h\le_T a^{-1}h$ and $\tilde{b}^{-1}g\in b^{-1}h$, statement 3 yields the existence of $\tilde{a}'\in a^{-1}h$ such that $\tilde{a}'\ge_R\tilde{b}^{-1}g$. We then define $\tilde{a}\coloneqq g\cdot(\tilde{a}')^{-1}$.

\item[$4\Rightarrow 5$] Let $a\in[1,h]_T$, and fix $\tilde{a}\in a$, $(t_1,\dots,t_k),(t_1',\dots,t_k')\in\Red_T(a)$. We shall prove the existence of $\tau\in\Br_k$ such that $(t_1',\dots,t_k')=\tau\cdot(t_1,\dots,t_k)$.

For each $j\in\{1,\dots,k\}$, let $a_j\coloneqq t_1t_2\dots t_j$ and $a_j'\coloneqq t_1' t_2'\dots t_j'$. From $\tilde{a}\in a=a_k=a_k'$ and statement 4, it is easy to prove by induction that for each $j\in\{1,\dots,k\}$ there exist $\tilde{a}_j\in a_j$ and $\tilde{a}_j'\in a_j'$ such that $\tilde{a}_j\le_R\tilde{a}_{j+1}$ and $\tilde{a}_j'\le_R\tilde{a}_{j+1}'$. We then define $\tilde{a}_0=\tilde{a}_0'\coloneqq 1\in[1,g]_R$, and for each $j\in\{1,\dots,k\}$ we define $\tilde{t}_j\coloneqq\tilde{a}_{j-1}^{-1}\tilde{a}_j$ and $\tilde{t}_j'\coloneqq(\tilde{a}_{j-1}')^{-1}\tilde{a}_j'$. Notice that for each $j\in\{1,\dots,k\}$ we have $\tilde{t}_j\in t_j$ and $\tilde{t}_j'\in t_j'$.

We have $(\tilde{t}_1,\dots,\tilde{t}_n),(\tilde{t}_1',\dots,\tilde{t}_n')\in\Red_R(\tilde{a})$. By Theorem \ref{hurwitz-fg}, we know that there exists $\tau\in\Br_k$ such that $(\tilde{t}_1',\dots,\tilde{t}_k')=\tau\cdot(\tilde{t}_1,\dots,\tilde{t}_k)$. Applying the group morphism $\pi$ to all elements of these $k$-tuples yields $(t_1',\dots,t_k')=\tau\cdot(t_1,\dots,t_k)$.

\item[$5\Rightarrow 1$] Let $t,b,\tilde{b}$ as in statement 1. Fix $(\tilde{t}_1',\dots,\tilde{t}_k')\in\Red_R(\tilde{b})$ and $(t_1,\dots,t_{k-1})\in\Red_T(bt^{-1})$. For each $j\in\{1,\dots,k\}$, let $t_j'\coloneqq\pi(\tilde{t}_j')$. Statement 5 yields the existence of $\tau\in\Br_k$ such that $\tau\cdot(t_1',\dots,t_k')=(t_1,\dots,t_{k-1},t)$. We then define $\tilde{t}$ as the last element of the $k$-tuple $\tau\cdot(\tilde{t}_1',\dots,\tilde{t}_k')$.
\end{enumerate}
\end{proof}

\subsection{An action of $\Br_n$ on $[1,g]_R$}
There is a well-known isomorphism between $\Br_n$ and the group of homeomorphisms from $\C\setminus\{x_1,\dots,x_n\}$ to itself, up to isotopy, that are pointwise equal to the identity map in the complement of a compact subset of $\C$. This isomorphism yields an action of $\Br_n$ by automorphisms on the fundamental group of $\C\setminus\{x_1,\dots,x_n\}$, which is $F_n$. We shall denote this action with the symbol $\star$, in order to distinguish it with the Hurwitz action, which we shall keep denoting with $\cdot$. We shall denote the map $F_n\rightarrow F_n$ given by $a\mapsto\tau\star a$ as $\iota(\tau)$.

Fixing $i\in\{1,\dots,n-1\}$, there are two ways to see $\iota(\sigma_i)$. First, it can be seen as the map between fundamental groups induced by the homeomorphism $\alpha_i$ from $\C\setminus\{x_1,\dots,x_n\}$ to itself described in Figure \ref{fig-disegnino}.
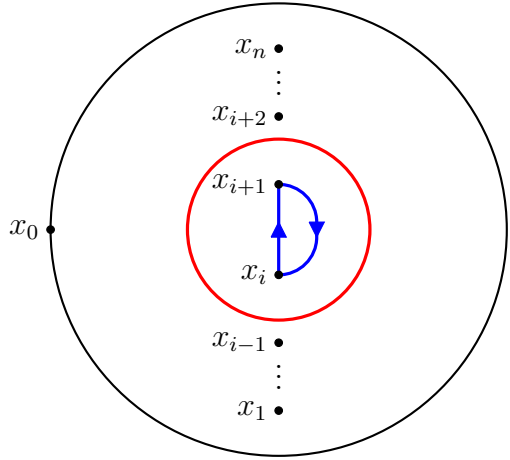
\begin{figure}[h]
\centering

\begin{tikzpicture}
\draw [thick] (0,0) circle [radius=3];
\draw [fill] (-3,0) circle [radius=0.05];
\node[left] at (-3,0) {$x_0$};
\draw [fill] (0,-1.5) circle [radius=0.05];
\node[left] at (0,-1.5) {$x_{i-1}$};
\node at (0,-1.85) {$\vdots$};
\draw [fill] (0,-2.4) circle [radius=0.05];
\node[left] at (0,-2.4) {$x_1$};
\draw [fill] (0,1.5) circle [radius=0.05];
\node[left] at (0,1.5) {$x_{i+2}$};
\node at (0,2.05) {$\vdots$};
\draw [fill] (0,2.4) circle [radius=0.05];
\node[left] at (0,2.4) {$x_n$};
\draw [red,very thick] (0,0) circle [radius=1.2];
\draw [blue,very thick] (0,-0.6) -- (0,0.6) arc [radius=0.5,start angle=90,end angle=0] -- (0.5, -0.1) arc [radius=0.5,start angle=0,end angle=-90];
\draw [fill] (0,-0.6) circle [radius=0.05];
\node[left] at (0,-0.6) {$x_{i}$};
\draw [fill] (0,0.6) circle [radius=0.05];
\node[left] at (0,0.6) {$x_{i+1}$};
\draw [blue,fill=blue] (-0.1,-0.1) -- (0.1,-0.1) -- (0,0.1) -- (-0.1,-0.1);
\draw [blue,fill=blue] (0.4,0.1) -- (0.6,0.1) -- (0.5,-0.1) -- (0.4,0.1);
\end{tikzpicture}
\caption{The black circle is the image of a noncrossing loop that represents $g$. Let $F:\C\times[0,1]\rightarrow \C$ be an isotopy such that $F(\cdot,0)$ is the identity, the restriction of $F(\cdot,t)$ outside of the disk delimited by the red circle is the identity for each $t$, and the blue paths are $F(x_i,\cdot)$ and $F(x_{i+1},\cdot)$. The homeomorphism $\alpha_i$ is the restriction of $F(\cdot,1)$ to $\C\setminus\{x_1,\dots,x_n\}$.}
\label{fig-disegnino}
\end{figure}
From this, we can prove the following:
\begin{prop}
\label{prop-nc-rainbow}
For each $\tau\in\Br_n$, the map $\iota(\tau):F_n\rightarrow F_n$ restricts to a poset isomorphism $(\NC,\subseteq)\rightarrow (\NC,\subseteq)$.
\end{prop}
\begin{proof}
We shall prove the statement assuming $\tau=\sigma_i$ for some $i\in\{1,\dots,n-1\}$. The case $\tau=\sigma_i^{-1}$ is similar, and the general case follows by these two by induction on $\len_{\{\sigma_1,\dots,\sigma_{n-1}\}}(\tau)$.

Let $a\in\NC$, let $\gamma$ be a noncrossing loop representing $a$, and let $\alpha_i$ be the homeomorphism described in Figure \ref{fig-disegnino}. Since $\alpha_i$ is an orientation-preserving homeomorphism, we know that $\alpha_i\circ\gamma$ does not self-intersect and is followed counterclockwise, so it is a noncrossing loop. Therefore, we have $\sigma_i\star a = (\alpha_i)_*([\gamma]) = [\alpha_i\circ\gamma]\in\NC$. This proves that $\iota(\sigma_i)$ restricts to a map $\NC\rightarrow\NC$.

Let $b,c\in\NC$ such that $b\subseteq c$. By definition, there exist let $\delta,\varepsilon$ noncrossing loops representing $b,c$ respectively, such that $\delta\subseteq\varepsilon$. Since $\alpha_i$ is an homeomorphism, we know that $\alpha_i\circ\delta\subseteq\alpha_i\circ\varepsilon$, which yields $\sigma_i\star b\subseteq\sigma_i\star c$. This proves that $\iota(\sigma_i)$ restricts to a poset morphism $(\NC,\subseteq)\rightarrow (\NC,\subseteq)$.

In a similar way, one can prove the same for $\iota(\sigma_i^{-1})$, and it is easy to check that the restrictions of $\iota(\sigma_i)$ and $\iota(\sigma_i^{-1})$ on $\NC$ are inverse to one another. This proves that $\iota(\sigma_i)$ restricts to a poset isomorphism $(\NC,\subseteq)\rightarrow (\NC,\subseteq)$.
\end{proof}
\begin{cor}
\label{cor-arconc}
For each $\tau\in\Br_n$, the map $\iota(\tau):F_n\rightarrow F_n$ restricts to a poset isomorphism $([1,g]_R,\le_R)\rightarrow ([1,g]_R,\le_R)$.
\end{cor}
\begin{proof}
Let $i\in\{1,\dots,n-1\}$, and let $\alpha_i$ be the homeomorphism described in Figure \ref{fig-disegnino}. From that figure it is evident that $(\alpha_i)_*(g)=g$, and thus $\sigma_i\star g = g$. Similarly to the proof of Proposition \ref{prop-nc-rainbow}, this is enough to prove $\tau\star g = g$.

Recalling the definition of $\NC_g$, from $\tau\star g = g$ and Proposition \ref{prop-nc-rainbow} we get that $\iota(\tau)$ restricts to a poset isomorphism $(\NC_g,\subseteq)\rightarrow (\NC_g,\subseteq)$. The statement follows from Theorem \ref{ncg-equal-1g}.
\end{proof}

The map $\iota(\sigma_i)$ can also be defined in a purely algebraic way, as follows:
\begin{equation}
\sigma_i\star f_j=\begin{cases}
f_{i+1}&\text{if }j=i\\
f_{i+1}^{-1}f_if_{i+1}&\text{if }j=i+1\\
f_j&\text{otherwise}
\end{cases}
\label{eqn-def-star}
\end{equation}
If $k$ is a positive integer, we define an action of $\Br_n$ on $(F_n)^k$ as follows:
\[\tau\star (a_1,\dots,a_k)\coloneqq(\tau\star a_1,\dots,\tau\star a_k)\]
We shall now explain the relationship between this action and the Hurwitz action.
\begin{prop}
\label{prop-starhur-1}
For each $\beta,\tau\in\Br_n$ and for each $(a_1,\dots,a_n)\in (F_n)^n$ we have $\beta\star(\tau\cdot(a_1,\dots,a_n))=\tau\cdot(\beta\star(a_1,\dots,a_n))$
\end{prop}
\begin{proof}
This follows immediately from the fact that $\iota(\beta)$ is a group morphism.
\end{proof}
\begin{prop}
\label{prop-starhur-2}
For each $i\in\{1,\dots,n-1\}$ and for each $\varepsilon\in\{-1,1\}$, we have $\sigma_i^\varepsilon\star(r_1,\dots,r_n)=\sigma_i^{-\varepsilon}\star(r_1,\dots,r_n)$
\end{prop}
\begin{proof}
If $\varepsilon=1$, this follows immediately from \eqref{eqn-def-star}. Moreover, from \eqref{eqn-def-star} we have:
\[
\sigma_i^{-1}\star f_j=\begin{cases}
f_if_{i+1}f_i^{-1}&\text{if }j=i\\
f_i&\text{if }j=i+1\\
f_j&\text{otherwise}
\end{cases}
\]
This yields the statement in the case $\varepsilon=-1$.
\end{proof}
\begin{cor}
\label{cor-starhur}
For each $\beta,\tau\in\Br_n$ we have $\beta\star(\tau\cdot(f_1,\dots,f_n))=\tau\beta^{-1}\cdot(f_1,\dots,f_n)$.
\end{cor}
\begin{proof}
We proceed by induction on $\ell\coloneqq\len_{\{\sigma_1,\dots,\sigma_{n-1}\}}(\beta)$. If $\ell=0$, then $\beta=1$ and the statement is trivial. Therefore, we shall assume $\ell>0$ and $\beta=\rho\beta'$, where $\rho\in\{\sigma_i,\sigma_i^{-1}\}$ for some $i\in\{1,\dots,n-1\}$ and $\len_{\{\sigma_1,\dots,\sigma_{n-1}\}}(\beta')=\ell-1$. By the induction hypothesis, we have $\beta'\star(\tau\cdot(f_1,\dots,f_n))=\tau(\beta')^{-1}\cdot(f_1,\dots,f_n)$. Therefore, Propositions \ref{prop-starhur-1} and \ref{prop-starhur-2} yield:
\[
\begin{split}
\beta\star(\tau\cdot(f_1,\dots,f_n))&= \rho\star(\beta'\star(\tau\cdot(f_1,\dots,f_n))) \\
&=\rho\star((\tau(\beta')^{-1})\cdot(f_1,\dots,f_n)) \\
&=\tau(\beta')^{-1}\cdot(\rho\star(f_1,\dots,f_n)) \\
&=\tau(\beta')^{-1}\cdot(\rho^{-1}\cdot(f_1,\dots,f_n)) \\
&=\tau(\beta')^{-1}\rho^{-1}\cdot(f_1,\dots,f_n) \\
&=\tau\beta^{-1}\cdot(f_1,\dots,f_n)
\end{split}
\]
\end{proof}
\begin{oss}
From Corollary \ref{cor-starhur}, we know that for each $\beta\in\Br_n$ we have $\beta\star(f_1,\dots,f_n)=\beta^{-1}\cdot(f_1,\dots,f_n)$. This is not true if we write any other reduced $R$-word representing $g$ instead of $(f_1,\dots,f_n)$.
\end{oss}
\subsection{Conditions related to the $\star$ action}
In this Subsection we shall focus on the action of $H$ on $F_n$ via the restriction of $\star$. From Corollary \ref{cor-arconc}, we know that this restricts to an action of $H$ on $[1,g]_R$. \begin{prop}
\label{prop-h-ok}
For each $\beta\in H$ and for each $\tilde{a}\in [1,g]_R$ we have $\pi(\tilde{a})=\pi(\beta\star \tilde{a})$.
\end{prop}
\begin{proof}
Let $k\coloneqq\len_R(\tilde{a})$, let $(r_1,\dots,r_k)\in\Red_R(\tilde{a})$, and let $(r_{k+1},\dots,r_n)\in\Red_R(\tilde{a}^{-1}g)$. We then have $(r_1,\dots,r_n)\in\Red_R(g)$. By Theorem \ref{hurwitz-fg}, there exists $\tau\in\Br_n$ such that $(r_1,\dots,r_n)=\tau\cdot (f_1,\dots,f_n)$. Moreover, notice that $\beta\star (r_1,\dots,r_k)\in\Red_R(\beta\star\tilde{a})$. Corollary \ref{cor-starhur} yields:
\[
\begin{split}
\beta\star(r_1,\dots,r_n)&=\beta\star(\tau\cdot(f_1,\dots,f_n))\\
&= \tau\beta^{-1}\cdot(f_1,\dots,f_n)
\end{split}
\]
Therefore, $\pi(\beta\star\tilde{a})$ is the product of the first $k$ elements of $\tau\beta^{-1}\cdot(s_1,\dots,s_n)$. From $\beta\in H$ we get $\tau\beta^{-1}\cdot(s_1,\dots,s_n)=\tau\cdot(s_1,\dots,s_n)$, so $\pi(\beta\star\tilde{a})$ is the product of the first $k$ elements of $\tau\cdot(s_1,\dots,s_n)$, which yields $\pi(\beta\star\tilde{a})=\pi(\tilde{a})$.
\end{proof} 
From Proposition \ref{prop-h-ok}, we know that for each $a\in [1,h]_T$ the $\star$ action restricts to an action of $H$ on $[1,g]_R\cap\pi^{-1}(a)$. The main result of this Subsection is that $(W,T,h)$ is pan-transitive if and only if this action on $[1,g]_R\cap\pi^{-1}(a)$ is transitive for each $a\in [1,h]_T$. We start by proving one of the two implications, and then we shall prove the other one while stating other useful equivalent conditions.
\begin{prop}
\label{prop-rainbow-1}
Suppose that for each $\tilde{a},\tilde{a}'\in[1,g]_R$ such that $\pi(\tilde{a})=\pi(\tilde{a}')$ there exists $\tau\in H$ such that $\tau\star\tilde{a}'=\tilde{a}$. Then, $(W,T,h)$ is pan-transitive.
\end{prop}
\begin{proof}Fix $a,b\in[1,h]_T$ such that $a\le_T b$, and fix $\tilde{a}\in a$. From Proposition \ref{poset-morph-free-cox}, there exist $\tilde{a}'\in a$ and $\tilde{b}'\in b$ such that $\tilde{a}'\le_R\tilde{b}'$. From our assumption, there exists $\tau\in H$ such that $\tau\star\tilde{a}'=\tilde{a}$. We then define $\tilde{b}\coloneqq\tau\star\tilde{b}'$, and notice that Corollary \ref{cor-arconc} yields $\tilde{a}\le_R\tilde{b}$. This proves the third statement from Proposition \ref{ele-sse-1gt}, which is equivalent to $(W,T,h)$ being pan-transitive.
\end{proof}

Let $k$ be a positive integer. We say that a $k$-tuple $(p_1,\dots,p_k)\in (\NC_g)^k$ is a \emph{noncrossing $k$-tuple} if there are noncrossing loops $\gamma_1,\dots,\gamma_k$ representing $p_1,\dots,p_k$ respectively that do not intersect each other (except in $x_0$). We say that two noncrossing $k$-tuples $(p_1,\dots,p_k),(p_1',\dots,p_k')$ are \emph{equivalent} if for each $i\in\{1,\dots,k\}$ we have $\pi(p_i)=\pi(p_i')$.

\begin{lemma}
\label{lemma-noncr-couple}
Let $(p_1,p_2)$ be a noncrossing couple, and assume that $(W,T,h)$ is pan-transitive. Then, we have $p_1\subseteq p_2$ if and only if we have $\pi(p_1)\le_T\pi(p_2)$.
\end{lemma}
\begin{proof}
The implication $p_1\subseteq p_2\Rightarrow\pi(p_1)\le_T\pi(p_2)$ comes from Theorem \ref{ncg-equal-1g} and Proposition \ref{poset-morph-free-cox}. To prove the other implication, we shall define $a\coloneqq\pi(p_1),b\coloneqq\pi(p_2)$ and assume $a\le_Tb$.

Suppose $p_1\supseteq p_2$. From Theorem \ref{ncg-equal-1g} and Proposition \ref{poset-morph-free-cox}, we get $a\ge_Tb$, which yields $a=b$, and thus $p_1=p_2$, which yields $p_1\subseteq p_2$.

Assume by contradiction $p_1\not\subseteq p_2$. We then also have $p_1\not\supseteq p_2$. Since $(p_1,p_2)$ is a noncrossing couple, we then have $p_1p_2\in\NC_g$ or $p_2p_1\in\NC_g$. Applying the fourth statement from Proposition \ref{ele-sse-1gt}, we know that there exists $p_3\subseteq p_2$ such that $p_3\in a$. Therefore, we have $p_1p_3\in\NC_g$ or $p_3p_1\in\NC_g$. Either way, Proposition \ref{poset-morph-free-cox} yields $a^2\in [1,h]_T$, and the height of $a^2$ in $[1,h]_T$ is $\len_R(p_1)+\len_R(p_3)=2\len_T(a)$. Assuming $a\ne 1$ is a contradiction, because fixing $(t_1,\dots,t_k)\in\Red_T(a)$, we have that $(t_1,\dots,t_{k-1},t_kt_1t_k,t_kt_2t_k,\dots,t_kt_{k-1}t_k)$ is a $T$-word representing $a^2$ of length $2\len_T(a)-2$. Therefore, we have $a=1$, which yields $p_1=1$, which contradicts $p_1\not\subseteq p_2$.
\end{proof}

We say that a noncrossing $k$-tuple $(p_1,\dots,p_k)$ is \emph{well-ordered} if for each $i<j$ we have $p_i\subseteq p_j$ or $p_ip_j\in\NC_g$. We say that a noncrossing $k$-tuple $(p_1,\dots,p_k)$ is \emph{without inclusions} if for each $i<j$ such that $p_i\ne 1$ and $p_j\ne 1$ we have $p_i\not\subseteq p_j$ and $p_i\not\supseteq p_j$.

Let $(p_1,\dots,p_k),(p_1',\dots,p_k')$ be two noncrossing $k$-tuples. We define the partial order $\preceq$ on $\{1,\dots,k\}$ such that $i\preceq j$ if and only if one of the following holds:
\begin{itemize}
\item $i=j$
\item $p_i\subset p_j$
\item $p_i p_j\in\NC_g$ and $p_j\ne 1$
\end{itemize}
Moreover, we define $\preceq'$ in the same way, using $(p_1',\dots,p_k')$ instead of $(p_1,\dots,p_k)$. We say that $(p_1,\dots,p_k),(p_1',\dots,p_k')$ are \emph{ordered alike} if the relations $\preceq,\preceq'$ are equal. Notice that, if two noncrossing $k$-tuples are ordered alike, we can apply the same permutation to those $k$-tuples in order to make them both well-ordered.
\begin{prop}
\label{prop-rainbow-tfae}
The following are equivalent:
\begin{enumerate}
\item $(W,T,h)$ is pan-transitive,
\item For each couple of noncrossing $k$-tuples $(p_1,\dots,p_k),(p_1',\dots,p_k')$ that are equivalent, well-ordered and without inclusions, there exists $\beta\in H$ such that $\beta\star (p_1',\dots,p_k')=(p_1,\dots,p_k)$,
\item For each couple of noncrossing $k$-tuples $(p_1,\dots,p_k),(p_1',\dots,p_k')$ that are equivalent and well-ordered, there exists $\beta\in H$ such that $\beta\star (p_1',\dots,p_k')=(p_1,\dots,p_k)$,
\item For each couple of noncrossing $k$-tuples $(p_1,\dots,p_k),(p_1',\dots,p_k')$ that are equivalent and ordered alike, there exists $\beta\in H$ such that $\beta\star (p_1',\dots,p_k')=(p_1,\dots,p_k)$,
\item For each $\tilde{a},\tilde{a}'\in[1,g]_R$ such that $\pi(\tilde{a})=\pi(\tilde{a}')$ there exists $\tau\in H$ such that $\tau\star\tilde{a}'=\tilde{a}$.
\end{enumerate}
\end{prop}
\begin{proof}
The implication $3\Rightarrow 4$ follows from the fact that, given two noncrossing $k$-tuples that are ordered alike, we can apply the same permutation to those $k$-tuples in order to make them both well-ordered. The implication $4\Rightarrow 5$ is just the restriction to the special case $k=1$. The implication $5\Rightarrow 1$ is Proposition \ref{prop-rainbow-1}. To conclude the proof, we shall prove the implications $1\Rightarrow 2$ and $2\Rightarrow 3$.
\begin{enumerate}
\item[$1\Rightarrow 2$] Fix two noncrossing $k$-tuples $(p_1,\dots,p_k),(p_1',\dots,p_k')$ that are equivalent, well-ordered and without inclusions. Define $p_{k+1}\coloneqq (p_1\dots p_k)^{-1}g$, $p_{k+1}'\coloneqq (p_1'\dots p_k')^{-1}g$.

Fix $i\in\{1,\dots,k+1\}$, and define $n_i\coloneqq\len_R(p_i)=\len_R(p_i')$ (the equality holds because $(p_1,\dots,p_k),(p_1',\dots,p_k')$ are equivalent and because of Proposition \ref{poset-morph-free-cox}), and fix $(r_{i,1},\dots,r_{i,n_i})\in\Red_R(p_i)$ and $(r_{i,1}'',\dots,r_{i,n_i}'')\in\Red_R(p_i')$. Since $(W,T,h)$ is pan-transitive, there exists $\beta_i\in\Br_{n_i}$ such that $\beta_i\cdot(\pi(r_{i,1}''),\dots,\pi(r_{i,n_i}''))=(\pi(r_{i,1}),\dots,$ $\pi(r_{i,n_i}))$. Define $(r_{i,1}',\dots,r_{i,n_i}')=\beta_i\cdot(r_{i,1}'',\dots,r_{i,n_i}'')$.

Notice that we have:
\[(r_{1,1},\dots,r_{1,n_1},r_{2,1},\dots,r_{2,n_2},\dots,r_{k+1,1},\dots,r_{k+1,n_{k+1}})\in\Red_R(g)\]
\[(r_{1,1}',\dots,r_{1,n_1}',r_{2,1}',\dots,r_{2,n_2}',\dots,r_{k+1,1}',\dots,r_{k+1,n_{k+1}}')\in\Red_R(g)\]
Theorem \ref{hurwitz-fg} yields the existence of $\tau,\tau'\in\Br_n$ such that:
\[(r_{1,1},\dots,r_{1,n_1},r_{2,1},\dots,r_{2,n_2},\dots,r_{k+1,1},\dots,r_{k+1,n_{k+1}})=\tau\cdot(f_1,\dots,f_n)\]
\[(r_{1,1}',\dots,r_{1,n_1}',r_{2,1}',\dots,r_{2,n_2}',\dots,r_{k+1,1}',\dots,r_{k+1,n_{k+1}}')=\tau'\cdot(f_1,\dots,f_n)\]
Corollary \ref{cor-starhur} yields:
\[(r_{1,1},\dots,r_{1,n_1},r_{2,1},\dots,r_{2,n_2},\dots,r_{k+1,1},\dots,r_{k+1,n_{k+1}})=\tau^{-1}\star(f_1,\dots,f_n)\]
\[(r_{1,1}',\dots,r_{1,n_1}',r_{2,1}',\dots,r_{2,n_2}',\dots,r_{k+1,1}',\dots,r_{k+1,n_{k+1}}')=(\tau')^{-1}\star(f_1,\dots,f_n)\]
Therefore:
\[\begin{split}
\tau^{-1}\tau'\,\star\,&(r_{1,1}',\dots,r_{1,n_1}',r_{2,1}',\dots,r_{2,n_2}',\dots,r_{k+1,1}',\dots,r_{k+1,n_{k+1}}')=\\
&(r_{1,1},\dots,r_{1,n_1},r_{2,1},\dots,r_{2,n_2},\dots,r_{k+1,1},\dots,r_{k+1,n_{k+1}})
\end{split}\]
Which yields $\tau^{-1}\tau'\star p_i'=p_i$ for each $i\in\{1,\dots,k+1\}$. We conclude by proving $\tau^{-1}\tau'\in H$:
\[\begin{split}
\!\!\!\!\!\!\!\!\!\!\!\!\!\!\!\!\!\!
\tau^{-1}\tau'\cdot(s_1,\dots,s_n)=&\tau^{-1}\cdot\left(\pi(r_{1,1}'),\dots,\pi(r_{1,n_1}'),\pi(r_{2,1}'),\dots,\pi(r_{2,n_2}'),\dots,\pi(r_{k+1,1}'),\dots,\pi(r_{k+1,n_{k+1}}')\right)\\
=&\tau^{-1}\cdot\left(\pi(r_{1,1}),\dots,\pi(r_{1,n_1}),\pi(r_{2,1}),\dots,\pi(r_{2,n_2}),\dots,\pi(r_{k+1,1}),\dots,\pi(r_{k+1,n_{k+1}})\right)\\
=&(s_1,\dots,s_n)
\end{split}\]
Where the equality between the first two rows follows from the definition of the $t_{i,j}'$.
\item[$2\Rightarrow 3$] Fix two noncrossing $k$-tuples $(p_1,\dots,p_k),(p_1',\dots,p_k')$ that are equivalent and well-ordered, and define by induction:
\[
\begin{cases}
q_1=p_1\\
v_i=\prod_{j\in\{1,\dots,i-1\}\text{ such that } q_j\subseteq p_i}q_j\\
q_i=v_i^{-1}p_i
\end{cases}
\]
We define the $q_i'$ and the $v_i'$ in the same way, using $(p_1',\dots,p_k')$ instead of $(p_1,\dots,p_k)$. It is easy to prove by induction on $i$ that $(q_1,\dots,q_i,p_{i+1},\dots,p_k)$ and $(q_1',\dots,q_i',p_{i+1}',$ $\dots,p_k')$ are noncrossing $k$-tuples that are well-ordered, and that $(q_1,\dots,q_i)$ and $(q_1',\dots,q_i')$ are without inclusions. We shall prove that they are also equivalent. If $i=0$ we have nothing to prove, so we shall assume $i>0$ and that $(q_1,\dots,q_{i-1},p_{i},\dots,p_k)$ and $(q_1',\dots,q_{i-1}',p_{i}',\dots,p_k')$ are equivalent noncrossing $k$-tuples that are well-ordered, and that $(q_1,\dots,q_{i-1})$ and $(q_1',\dots,q_{i-1}')$ are without inclusions. From Lemma \ref{lemma-noncr-couple} we have:
\[v_i=\prod_{j\in\{1,\dots,i-1\}\text{ such that } \pi(q_j)\le_T \pi(p_i)}q_j\]
Given that $(q_1,\dots,q_{i-1},p_{i},\dots,p_k)$ and $(q_1',\dots,q_{i-1}',p_{i}',\dots,p_k')$ are equivalent, this yields $\pi(v_i)=\pi(v_i')$, hence $\pi(q_i)=\pi(q_i')$, so $(q_1,\dots,q_i,p_{i+1},\dots,p_k)$ and $(q_1',\dots,q_i',$ $p_{i+1}',\dots,p_k')$ are equivalent.

Taking $i=k$, we get that $(q_1,\dots,q_k)$ and $(q_1',\dots,q_k')$ are equivalent noncrossing $k$-tuples that are well-ordered and without inclusions. Therefore, there exists $\beta\in H$ such that $\beta\star (q_1',\dots,q_k')=(q_1,\dots,q_k)$. Since $\iota(\beta)$ is a group morphism, this yields $\beta\star (p_1',\dots,p_k')=(p_1,\dots,p_k)$.
\end{enumerate}
\end{proof}

\section{Free products}
\label{sect-fg}
Throughout Section \ref{sect-fg}, we fix two Coxeter systems $(W_1,S_1),(W_2,S_2)$, we fix $T_1,T_2$ their respective sets of reflections, we fix an $n$-tuple $(s_1,\dots,s_n)$ such that $S_1=\{s_1,\dots,s_k\}$ and $S_2=\{s_{k+1},\dots,s_n\}$, and we fix the Coxeter elements $h_1=s_1\dots s_k$, $h_2=s_{k+1}\dots s_n$ for $(W_1,S_1),(W_2,S_2)$ respectively. We then define:
\[H_1\coloneqq H(W_1,(s_1,\dots,s_k)) \qquad H_2\coloneqq H(W_2,(s_{k+1},\dots,s_n))\]
\[\overline{H}_1\coloneqq \overline{H}(W_1,(s_1,\dots,s_k)) \qquad \overline{H}_2\coloneqq \overline{H}(W_2,(s_{k+1},\dots,s_n))\]

We then define $(W,S)$ as the Coxeter system such that $W=W_1*W_2$ and $S=S_1\sqcup S_2$, we define $T$ as its set of reflections, and we fix $h\coloneqq h_1 h_2 = s_1\dots s_n$ as its Coxeter element. We define:
\[H\coloneqq H(W,(s_1,\dots,s_n)) \qquad \overline{H}\coloneqq \overline{H}(W,(s_1,\dots,s_n))\]

We define $F_n,f_1,\dots,f_n,g,R$ as in Subsection \ref{subs-fg}, and we define the following group morphisms:
\[\begin{split}
\pi_1:F_n\longrightarrow&\, W_1 \\ 
f_i\longmapsto &\, \begin{cases}
s_i&\text{if }i\le k\\
1&\text{if }i> k\end{cases}
\end{split}
\qquad
\begin{split}
\pi_2:F_n\longrightarrow&\, W_2 \\ 
f_i\longmapsto &\, \begin{cases}
1&\text{if }i\le k\\
s_i&\text{if }i> k\end{cases}
\end{split}
\]\[
\begin{split}
\pi:F_n\longrightarrow&\, W \\ 
f_i\longmapsto &\, s_i\end{split}\]
Similarly to Section \ref{sect-ele}, if $a\in [1,h]_T$ and $\tilde{a}\in [1,g]_R\cap\pi^{-1}(a)$, we shall simply write $\tilde{a}\in a$. We shall also identify $F_n$ with the fundamental group of $(\C\setminus\{x_1,\dots,x_n\},x_0)$ and define noncrossing loops, $\NC$ and $\NC_g$ as in Subsection \ref{subs-fg}, in order to make use of the equality between $(\NC_g,\subseteq)$ and $([1,g]_R,\le_R)$ from Theorem \ref{ncg-equal-1g}.

Finally, we identify $\Br_k\times\Br_{n-k}$ with the subgroup of $\Br_n$ generated by $\sigma_1,\dots,\sigma_{k-1},$ $\sigma_{k+1},\dots,\sigma_{n-1}$ via the morphism defined by $(\sigma_i,1)\mapsto \sigma_i$ and $(1,\sigma_j)\mapsto \sigma_{k+j}$, and we identify each subgroup of $\Br_k\times\Br_{n-k}$ with its image via the same morphism.

\subsection{Free product of well-stabilized is well-stabilized}
In this Subsection we aim to prove that, if $(W_1,(s_1,\dots,s_k))$ and $(W_2,(s_{k+1},\dots,s_n))$ are well-stabilized, then so is $(W,(s_1,\dots,s_n))$. To achieve this, we first give a few results on $F_n$.
\begin{lemma}
\label{lemma-tfae-hh}
Let $\beta\in\Br_n$. The following are equivalent:
\begin{enumerate}
\item $\beta\in\Br_k\times \Br_{n-k}$,
\item The first $k$ elements of $\beta\cdot (f_1,\dots,f_n)$ belong to $\left<f_1,\dots,f_k\right> <F_n$ and the last $n-k$ elements of that $n$-tuple belong to $\left<f_{k+1},\dots,f_n\right> <F_n$.
\end{enumerate}
\end{lemma}
\begin{proof}
The implication $1\Rightarrow 2$ is immediate. We shall now show its converse.

Let $(r_1,\dots,r_n)\coloneqq\beta\cdot (f_1,\dots,f_n)$. From our assumption we get that $r_1\dots r_k$ is an element of $[1,g]_R$ of height $k$ which also belongs to $\left<f_1,\dots,f_k\right>$. The only such element is $f_1\dots f_k$, so we have $(r_1,\dots,r_k)\in\Red_{R\cap\left<f_1,\dots,f_k\right>}(f_1\dots f_k)$. Applying Theorem \ref{hurwitz-fg} to the free group $\left<f_1,\dots,f_k\right>$, there exists $\beta_1\in\Br_k$ such that $(r_1,\dots,r_k)=\beta_1\cdot(f_1,\dots,f_k)$. Similarly, there exists $\beta_2\in\Br_{n-k}$ such that $(r_{k+1},\dots,r_n)=\beta_1\cdot(f_{k+1},\dots,f_n)$. Identifying $\Br_k\times\Br_{n-k}$ with a subgroup of $\Br_n$ as described above, we then have that $\beta$ and $(\beta_1,\beta_2)$ act in the same way on $(f_1,\dots,f_n)$ via the Hurwitz action, which is free by Theorem \ref{hurwitz-fg}. This proves $\beta\in\Br_k\times\Br_{n-k}$.
\end{proof}

Since $F_n$ is the free product of its subgroups $\left<f_1,\dots,f_k\right>$ and $\left<f_{k+1},\dots,f_n\right>$, each element $a\in F_n$ has a unique factorization $a_1 a_2\dots a_N$, where each $a_i$ belongs to ${\cal{F}}\coloneqq\left<f_1,\dots,f_k\right>\cup\left<f_{k+1},\dots,f_n\right>\setminus\{1\}$ and for each $i<N$ we have $a_i \in \left<f_1,\dots,f_k\right>$ if and only if $a_{i+1} \in \left<f_{k+1},\dots,f_n\right>$. In other words, there exists a unique element $(a_1,a_2,\dots,a_N)\in\Red_{\cal{F}}(a)$. We say that $a$ is \emph{good} if for each $i\in\{1,\dots,N\}$ we have $\pi(a_i)\ne 1$.

\begin{lemma}
\label{lemma-nc-good}
Let $a\in [1,g]_R$. Then, $a$ is good.
\end{lemma}
\begin{proof}
If $k\in\{0,n\}$ or $a=1$, the statement is trivial. We shall assume $k\notin\{0,n\}$ and $a\ne 1$. Recall the equality of $(\NC_g,\subseteq)$ and $([1,g]_R,\le_R)$ from Theorem \ref{ncg-equal-1g}.

Let $(a_1,\dots,a_N)$ be the unique element of $\Red_{\cal{F}}(a)$. We shall prove that, for each $i\in\{1,\dots,N\}$, we have either $a_i\in\NC_g$ or $a_i^{-1}\in\NC_g$.

Fix noncrossing loops $\gamma,\gamma_g$ representing $a,g$ respectively. Up to homotopy, we may assume that $\gamma$ intersects the curves drawn in red in Figure \ref{fig-ra1} the minimum number of times.
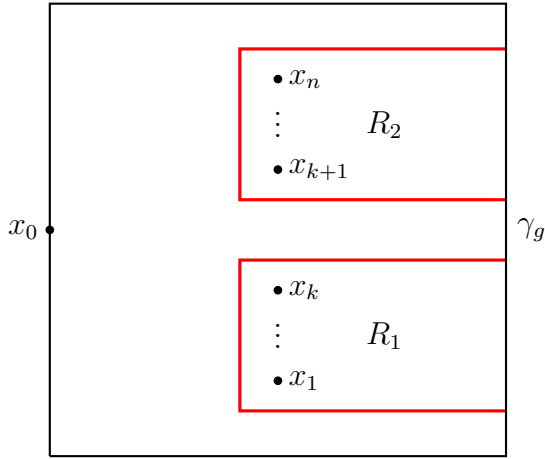
\begin{figure}[h]
\centering

\begin{tikzpicture}
\draw [fill] (-3,0) circle [radius=0.05];
\node[left] at (-3,0) {$x_0$};
\node[right] at (3,0) {$\gamma_g$};
\draw [fill] (0,-0.8) circle [radius=0.05];
\node[right] at (0,-0.8) {$x_{k}$};
\node at (0,-1.3) {$\vdots$};
\draw [fill] (0,-2) circle [radius=0.05];
\node[right] at (0,-2) {$x_1$};
\draw [fill] (0,0.8) circle [radius=0.05];
\node[right] at (0,0.8) {$x_{k+1}$};
\node at (0,1.5) {$\vdots$};
\draw [fill] (0,2) circle [radius=0.05];
\node[right] at (0,2) {$x_n$};
\node at (1.4,-1.4) {$R_1$};
\node at (1.4,1.4) {$R_2$};
\draw [red,very thick] (3,2.4) -- (-0.5,2.4) -- (-0.5,0.4) -- (3,0.4);
\draw [red,very thick] (3,-2.4) -- (-0.5,-2.4) -- (-0.5,-0.4) -- (3,-0.4);
\draw [thick] (-3,-3) -- (-3,3) -- (3,3) -- (3,-3) -- (-3,-3);
\end{tikzpicture}
\caption{We denote by $R_1,R_2$ the regions of $\C$ delimited by the image of $\gamma_g$ and by the red curves, containing $x_k, x_{k+1}$ respectively.}
\label{fig-ra1}
\end{figure}

Now, we define $\gamma_1,\dots,\gamma_{\tilde N}$ as the loops obtained from $\gamma$ following the procedure described in Figure \ref{fig-ra2},
\begin{figure}[h]
\centering

\begin{tikzpicture}[scale=0.7]
\draw [red,very thick] (3,2.7) -- (-0.8,2.7) -- (-0.8,0.4) -- (3,0.4);
\draw [red,very thick] (3,-2.7) -- (-0.8,-2.7) -- (-0.8,0.1) -- (3,0.1);
\draw [fill] (-3,0) circle [radius=0.05];
\node[left] at (-3,0) {$x_0$};
\node[right] at (3,0) {$\gamma_g$};
\draw [fill] (0,-2.15) circle [radius=0.05];
\node[right] at (0,-2.15) {$x_1$};
\draw [fill] (0,-1.3) circle [radius=0.05];
\node[right] at (0,-1.3) {$x_2$};
\draw [fill] (0,-0.45) circle [radius=0.05];
\node[right] at (0,-0.45) {$x_3$};
\draw [fill] (0,1.15) circle [radius=0.05];
\node[right] at (0,1.15) {$x_4$};
\draw [fill] (0,2) circle [radius=0.05];
\node[right] at (0,2) {$x_5$};

\draw [thick] (-0.8,-2.5) to [out=180,in=-70] (-3,0) to [out=70,in=180] (-0.8,2.45);
\draw [thick] (-0.8, -1.8) arc [x radius = 1.5, y radius = 1.65, start angle=270, end angle=90];
\draw [thick] (-0.8, -1.65) arc [x radius = 1.1, y radius = 1.225, start angle=270, end angle=90];
\draw [thick] (-0.8, -0.95) arc [x radius = 0.6, y radius = 0.8, start angle=270, end angle=90];

\node[right] at (2.1,-1.3) {$\gamma$};

\draw [thick] (-0.8,-0.95) -- (0.7,-0.95) arc [radius=0.35, start angle=90, end angle=-90] -- (-0.8,-1.65);
\draw [thick] (0,-0.8) arc [radius=0.35, start angle=270, end angle=90] -- (1.2,-0.1) arc [x radius=1, y radius=1.2, start angle=90, end angle=-90] -- (-0.8,-2.5);
\draw [thick] (0,-0.8) -- (1.1,-0.8) arc [radius=0.5, start angle=90, end angle=-90] -- (-0.8,-1.8);
\draw [thick] (-0.8,1.5) -- (0.7,1.5) arc [radius=0.35, start angle=90, end angle=-90] -- (-0.8,0.8);
\draw [thick] (-0.8,2.45) -- (1.1,2.45) arc [radius=0.9, start angle=90, end angle=-90] -- (-0.8,0.65);
\draw [fill] (0.7,-0.1) -- (0.9,0) -- (0.9,-0.2) -- (0.7,-0.1);
\draw [fill] (-0.6,-0.95) -- (-0.4,-0.85) -- (-0.4,-1.05) -- (-0.6,-0.95);
\draw [fill] (0.7,2.45) -- (0.9,2.55) -- (0.9,2.35) -- (0.7,2.45);
\draw [fill] (-0.3,1.5) -- (-0.5,1.6) -- (-0.5,1.4) -- (-0.3,1.5);

\draw [thick] (-3,-3) -- (-3,3) -- (3,3) -- (3,-3) -- (-3,-3);

\begin{scope}[shift={(9,0)}]
\draw [red,very thick] (3,2.7) -- (-0.8,2.7) -- (-0.8,0.4) -- (3,0.4);
\draw [red,very thick] (3,-2.7) -- (-0.8,-2.7) -- (-0.8,0.1) -- (3,0.1);
\draw [fill] (-3,0) circle [radius=0.05];
\node[left] at (-3,0) {$x_0$};
\node[right] at (3,0) {$\gamma_g$};
\draw [fill] (0,-2.15) circle [radius=0.05];
\node[right] at (0,-2.15) {$x_1$};
\draw [fill] (0,-1.3) circle [radius=0.05];
\node[right] at (0,-1.3) {$x_2$};
\draw [fill] (0,-0.45) circle [radius=0.05];
\node[right] at (0,-0.45) {$x_3$};
\draw [fill] (0,1.15) circle [radius=0.05];
\node[right] at (0,1.15) {$x_4$};
\draw [fill] (0,2) circle [radius=0.05];
\node[right] at (0,2) {$x_5$};

\draw [thick] (-0.8,-0.95) -- (0.7,-0.95) arc [radius=0.35, start angle=90, end angle=-90] -- (-0.8,-1.65);
\draw [thick] (0,-0.8) arc [radius=0.35, start angle=270, end angle=90] -- (1.2,-0.1) arc [radius=1.2, start angle=90, end angle=-90] -- (-0.8,-2.5);
\draw [thick] (0,-0.8) -- (1.1,-0.8) arc [radius=0.5, start angle=90, end angle=-90] -- (-0.8,-1.8);
\draw [thick] (-0.8,1.5) -- (0.7,1.5) arc [radius=0.35, start angle=90, end angle=-90] -- (-0.8,0.8);
\draw [thick] (-0.8,2.45) -- (1.1,2.45) arc [radius=0.9, start angle=90, end angle=-90] -- (-0.8,0.65);
\draw [fill] (0.7,-0.1) -- (0.9,0) -- (0.9,-0.2) -- (0.7,-0.1);
\draw [fill] (-0.6,-0.95) -- (-0.4,-0.85) -- (-0.4,-1.05) -- (-0.6,-0.95);
\draw [fill] (0.7,2.45) -- (0.9,2.55) -- (0.9,2.35) -- (0.7,2.45);
\draw [fill] (-0.3,1.5) -- (-0.5,1.6) -- (-0.5,1.4) -- (-0.3,1.5);

\draw [thick] (-3,-3) -- (-3,3) -- (3,3) -- (3,-3) -- (-3,-3);
\end{scope}

\begin{scope}[shift={(4.5,-7)}]
\draw [red,very thick] (3,2.7) -- (-0.8,2.7) -- (-0.8,0.4) -- (3,0.4);
\draw [red,very thick] (3,-2.7) -- (-0.8,-2.7) -- (-0.8,0.1) -- (3,0.1);
\draw [fill] (-3,0) circle [radius=0.05];
\node[left] at (-3,0) {$x_0$};
\node[right] at (3,0) {$\gamma_g$};
\draw [fill] (0,-2.15) circle [radius=0.05];
\node[right] at (0,-2.15) {$x_1$};
\draw [fill] (0,-1.3) circle [radius=0.05];
\node[right] at (0,-1.3) {$x_2$};
\draw [fill] (0,-0.45) circle [radius=0.05];
\node[right] at (0,-0.45) {$x_3$};
\draw [fill] (0,1.15) circle [radius=0.05];
\node[right] at (0,1.15) {$x_4$};
\draw [fill] (0,2) circle [radius=0.05];
\node[right] at (0,2) {$x_5$};

\draw [thick] (-0.8,-2.5) to [out=180,in=-70] (-3,0) to [out=70,in=180] (-0.8,2.45);
\draw [thick] (-0.8, -1.8) to [out=180,in=-60] (-3,0) to [out=55,in=180] (-0.8,1.5);
\draw [thick] (-0.8, -1.65) to [out=180,in=-55] (-3,0) to [out=40,in=180] (-0.8,0.8);
\draw [thick] (-0.8, -0.95) to [out=180,in=-40] (-3,0) to [out=35,in=180] (-0.8,0.65);

\node[right] at (2.1,-1.3) {$\gamma_1$};
\node[left] at (-0.7,-1.3) {$\gamma_3$};
\node[left] at (-0.7,1.1) {$\gamma_2$};
\node[right] at (1.9,1.55) {$\gamma_4$};

\draw [thick] (-0.8,-0.95) -- (0.7,-0.95) arc [radius=0.35, start angle=90, end angle=-90] -- (-0.8,-1.65);
\draw [thick] (0,-0.8) arc [radius=0.35, start angle=270, end angle=90] -- (1.2,-0.1) arc [x radius=1, y radius=1.2, start angle=90, end angle=-90] -- (-0.8,-2.5);
\draw [thick] (0,-0.8) -- (1.1,-0.8) arc [radius=0.5, start angle=90, end angle=-90] -- (-0.8,-1.8);
\draw [thick] (-0.8,1.5) -- (0.7,1.5) arc [radius=0.35, start angle=90, end angle=-90] -- (-0.8,0.8);
\draw [thick] (-0.8,2.45) -- (1.1,2.45) arc [radius=0.9, start angle=90, end angle=-90] -- (-0.8,0.65);
\draw [fill] (0.7,-0.1) -- (0.9,0) -- (0.9,-0.2) -- (0.7,-0.1);
\draw [fill] (-0.6,-0.95) -- (-0.4,-0.85) -- (-0.4,-1.05) -- (-0.6,-0.95);
\draw [fill] (0.7,2.45) -- (0.9,2.55) -- (0.9,2.35) -- (0.7,2.45);
\draw [fill] (-0.3,1.5) -- (-0.5,1.6) -- (-0.5,1.4) -- (-0.3,1.5);

\draw [thick] (-3,-3) -- (-3,3) -- (3,3) -- (3,-3) -- (-3,-3);
\end{scope}

\end{tikzpicture}
\caption{Example of construction of $\gamma_1,\dots,\gamma_{\tilde N}$. First, intersect the image of $\gamma$ with $R_1\cup R_2$. Then, complete each of the connected components of this intersection to a loop based in $x_0$, maintaining their orientation and in such a way that they do not self-intersect or intersect each other (except in $x_0$). Finally, number these loops such that $[\gamma]=[\gamma_1]\cdot[\gamma_2]\cdot\ldots\cdot[\gamma_{\tilde N}]$}
\label{fig-ra2}
\end{figure}
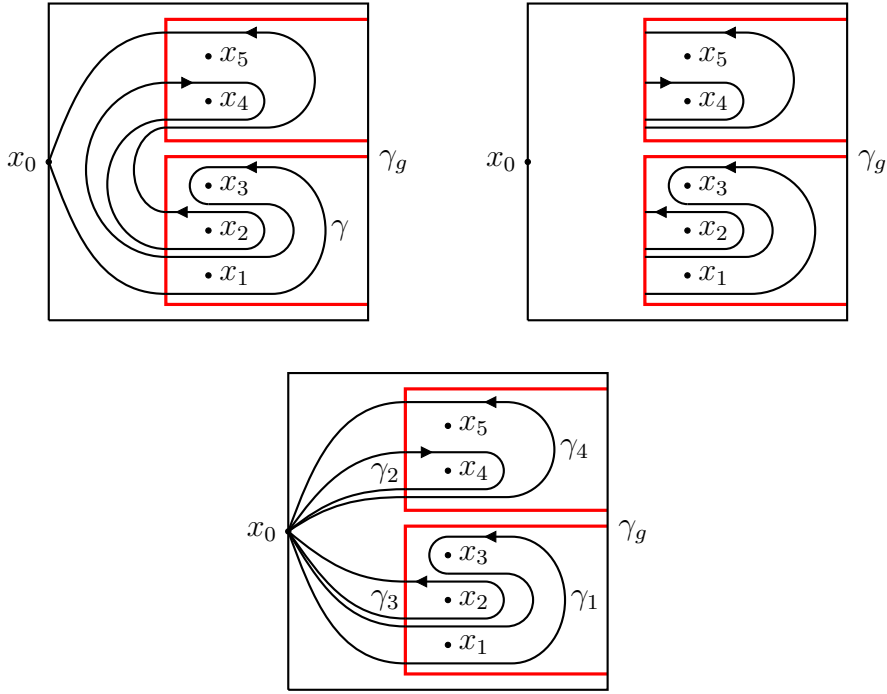
and we define $\tilde{a}_i\coloneqq [\gamma_i]\in F_n$ for each $i\in\{1,\dots\tilde N\}$. Since $\gamma_i$ does not self-intersect, we have either $\tilde{a}_i\in\NC_g$ or $\tilde{a}_i^{-1}\in\NC_g$. Notice that $\tilde{a}_1\dots\tilde{a}_{\tilde{N}}$ is a factorization fo $a$ in elements of $\cal F$ such that for each $i<\tilde N$ we have $\tilde{a}_i \in \left<f_1,\dots,f_k\right>$ if and only if $\tilde{a}_{i+1} \in \left<f_{k+1},\dots,f_n\right>$. So, $(\tilde{a}_1,\dots,\tilde{a}_{\tilde{N}})$ is the unique element of $\Red_{\cal F}(a)$, and is therefore equal to $(a_1,\dots,a_N)$. This proves that, for each $i\in\{1,\dots,N\}$, we have either $a_i\in\NC_g$ or $a_i^{-1}\in\NC_g$.

Now, for each $i\in\{1,\dots,N\}$, we define $b_i$ as the element of $\{a_i,a_i^{-1}\}\cap\NC_g$. Since $b_i\ne 1$, from Proposition \ref{poset-morph-free-cox} we have that $\pi(b_i)$ is an element of $[1,h]_T$ of positive height. In particular, it is not $1$, and thus $\pi(a_i)\ne 1$, which is the statement.
\end{proof}

\begin{lemma}
\label{lemma-oss-buono}
Let $a\in F_n$ be a good element. If $\pi(a)\in W_1$, then $a\in\left<f_1,\dots,f_k\right>$. If $\pi(a)\in W_2$, then $a\in\left<f_{k+1},\dots,f_n\right>$.
\end{lemma}
\begin{proof}
Let $(a_1,\dots,a_N)$ be the unique element of $\Red_{\cal{F}}(a)$. The factorization $\pi(a)=\pi(a_1)\pi(a_2)\dots\pi(a_n)$ is a factorization of $\pi(a)$ in elements of $W_1\cup W_2\setminus\{1\}$ such that for each $i<N$ the $i$-th factor belongs to $W_1$ if and only if the $(i+1)$-th factor belongs to $W_2$. Since $W$ is the free product of its subgroups $W_1$ and $W_2$, we have that $(\pi(a_1),\pi(a_2),\dots,\pi(a_n))$ is the unique reduced $(W_1\cup W_2\setminus\{1\})$-word for $\pi(a)$. Therefore, if $\pi(a)\in W_1\cup W_2$, we must have $N\le 1$, from which the statement follows easily.
\end{proof}

\begin{teo}
\label{thm-fg-hh}
We have $H=H_1\times H_2$ and $\overline{H}=\overline{H}_1\times \overline{H}_2$. In particular, if $(W_1,(s_1,\dots,s_k))$ and $(W_2,(s_{k+1},\dots,s_n))$ are well-stabilized, then so is $(W,(s_1,\dots,s_n))$.
\end{teo}
\begin{proof}
The crucial part of the proof is proving $H\subseteq\Br_k\times\Br_{n-k}$. From this and $\overline{H}\subseteq H$ it is easy to check the statement.

Let $\beta\in H$. In particular, we have that the first $k$ elements of $\beta\cdot (s_1,\dots,s_n)=(s_1,\dots,s_n)$ belong to $W_1$, while the last $n-k$ elements of that $n$-tuple belong to $W_2$. Let $(r_1,\dots,r_n)\coloneqq\beta\cdot (f_1,\dots,f_n)$. From Lemma \ref{lemma-nc-good}, we know that $r_1,\dots,r_n$ are good. From Lemma \ref{lemma-oss-buono}, we then know that $r_1,\dots,r_k$ belong to $\left<f_1,\dots,f_k\right>$, while $r_{k+1},\dots,r_n$ belong to $\left<f_{k+1},\dots,f_n\right>$. Lemma \ref{lemma-tfae-hh} then yields $\beta\in\Br_k\times\Br_{n-k}$.
\end{proof}

\subsection{Free product of pan-transitive is pan-transitive}
In this Subsection we aim to prove the following:
\begin{teo}
\label{thm-ele}
If $(W_1,T_1,h_1)$ and $(W_2,T_2,h_2)$ are pan-transitive, then so is $(W,T,h)$.
\end{teo}
The idea of our proof is to use the fact that $(W_1,T_1,h_1)$ and $(W_2,T_2,h_2)$ satisfy the fourth statement of Proposition \ref{prop-rainbow-tfae} and prove that $(W,T,h)$ satisfies the fifth one.

Let $a\in W$, and fix $\tilde{a}\in a$. From $\tilde{a}$, we construct $\gamma,\gamma_1,\dots,\gamma_N$ as in Figure \ref{fig-ra2}. For each $i\in\{1,\dots,N\}$ we define $\tilde{a}_i\coloneqq [\gamma_i]\in F_n$ and $a_i\coloneqq\pi(\tilde{a}_i)$. Notice that by Lemma \ref{lemma-nc-good} we have that $(a_1,\dots,a_N)$ is the element of $\Red_{W_1\cup W_2\setminus\{1\}}(a)$, therefore $N,a_1,\dots,a_N$ only depend on $a$.

For each $i\in\{1,\dots,N\}$, we shall define $\varepsilon_{\tilde{a}} (i)\in\{-1,1\}$ and $\overline{a}_i\in\NC_g$ such that $\overline{a}_i=\tilde{a}_i^{\varepsilon_{\tilde{a}} (i)}$. Moreover, we shall define $\overline{\gamma}_i$ as the noncrossing loop representing $\overline{a}_i$ that is equal to $\gamma_i$ if $\varepsilon_{\tilde{a}} (i)=1$, and is otherwise equal to $\gamma_i$ followed in reverse order.

We now define the total order relation $\preceq_{\tilde{a}}$ on $\{1,\dots,N\}$ such that $i\prec_{\tilde{a}}j$ if and only if one of the following holds:
\begin{itemize}
\item $\overline{\gamma}_i\subset\overline{\gamma}_j$,
\item $\overline{a}_i\ne\overline{a}_j$ and $(\overline{a}_i,\overline{a}_j)$ is a well-ordered noncrossing couple.
\end{itemize}
Notice that, if $a_i\subset a_j$, both of these conditions are met. We now state the following:
\begin{lemma}
\label{lemma-finalfg}
The function $\varepsilon_{\tilde a}:\{1,\dots,N\}\rightarrow\{-1,1\}$ and the order relation $\preceq_{\tilde a}$ only depend on $a$.
\end{lemma}
We shall now prove Theorem \ref{thm-ele} assuming Lemma \ref{lemma-finalfg}.
\begin{proof}
Let $a\in [1,h]_T\subseteq W\setminus\{1\}$, and let $\tilde{a},\tilde{a}'\in a$. Define $N,a_1,\dots,a_n,\tilde{a}_1,\dots,\tilde{a}_n,$ $\overline{a}_1,\dots,\overline{a}_n,\varepsilon_{\tilde{a}},\preceq_{\tilde{a}}$ as above, and define $\tilde{a}_1',\dots,\tilde{a}_n',\overline{a}_1',\dots,\overline{a}_n',\varepsilon_{\tilde{a}'},\preceq_{\tilde{a}'}$ in the same way using $\tilde{a}'$ instead of $\tilde{a}$. For the remainder of this proof, we shall assume $a_1\in W_1$ and that $N$ is even. The proof is basically the same in the other cases.

Notice that $(\overline{a}_1,\overline{a}_3,\dots,\overline{a}_{N-1})$ and $(\overline{a}_1',\overline{a}_3',\dots,\overline{a}_{N-1}')$ are equivalent noncrossing $\frac N2$-tuples, they are ordered alike because of Lemma \ref{lemma-finalfg}, and their elements belong to the subgroup $\left<f_1,\dots,f_k\right>\subseteq F_n$. Since $(W_1,T_1,h_1)$ satisfies the fourth statement of Proposition \ref{prop-rainbow-tfae}, there exists $\beta_1\in H_1$ such that $\beta_1\star (\overline{a}_1',\overline{a}_3',\dots,\overline{a}_{N-1}')=(\overline{a}_1,\overline{a}_3,\dots,\overline{a}_{N-1})$. Similarly, there exists $\beta_2\in H_2$ such that $\beta_2\star (\overline{a}_2',\overline{a}_4',\dots,\overline{a}_{N}')=(\overline{a}_2,\overline{a}_4,\dots,\overline{a}_{N})$. We then define $\beta=(\beta_1,\beta_2)\in H_1\times H_2=H$ (recall Theorem \ref{thm-fg-hh}), and notice that $\beta\star(\overline{a}_1',\overline{a}_2',\dots,\overline{a}_{N}')=(\overline{a}_1,\overline{a}_2,\dots,\overline{a}_{N})$. Therefore:
\[\begin{split}
\beta\star\tilde{a}'&=\beta\star(\tilde{a}_1'\tilde{a}_2'\dots\tilde{a}_{N}')\\
&=\beta\star\left((\overline{a}_1')^{\varepsilon_{\tilde{a}'}(1)}(\overline{a}_2')^{\varepsilon_{\tilde{a}'}(2)}\dots(\overline{a}_{N}')^{\varepsilon_{\tilde{a}'}(N)}\right)\\
&=\beta\star\left((\overline{a}_1')^{\varepsilon_{\tilde{a}}(1)}(\overline{a}_2')^{\varepsilon_{\tilde{a}}(2)}\dots(\overline{a}_{N}')^{\varepsilon_{\tilde{a}}(N)}\right)\\
&=(\beta\star\overline{a}_1')^{\varepsilon_{\tilde{a}}(1)}(\beta\star\overline{a}_2')^{\varepsilon_{\tilde{a}}(2)}\dots(\beta\star\overline{a}_{N}')^{\varepsilon_{\tilde{a}}(N)}\\
&=\overline{a}_1^{\varepsilon_{\tilde{a}}(1)}\overline{a}_2^{\varepsilon_{\tilde{a}}(2)}\dots\overline{a}_{N}^{\varepsilon_{\tilde{a}}(N)}=\tilde{a}
\end{split}\]
Where the equality between the second and third row follows from Lemma \ref{lemma-finalfg}. This proves that $(W,T,h)$ satisfies the fifth statement of Proposition \ref{prop-rainbow-tfae}, and is therefore pan-transitive.
\end{proof}
The remainder of this Subsection is devolved to the proof of Lemma \ref{lemma-finalfg}. To achieve this, we first define two auxiliary order relations.

We define the total order relation $\preceq_{\tilde{a}}'$ on $\{1,\dots,N\}$ such that $i\prec_{\tilde{a}}j$ if and only if one of the following holds:
\begin{itemize}
\item $\overline{\gamma}_i\subset\overline{\gamma}_j$ and $\varepsilon_{\tilde{a}}(j)=1$,
\item $\overline{\gamma}_i\supset\overline{\gamma}_j$ and $\varepsilon_{\tilde{a}}(i)=-1$,
\item $\overline{a}_i\ne\overline{a}_j$ and $(\overline{a}_i,\overline{a}_j)$ is a well-ordered noncrossing couple.
\end{itemize}
We then define the total order relation $\preceq_{\tilde{a}}''$ on $\{1,\dots,N\}$ such that $i\prec_{\tilde{a}}j$ if and only if one of the following holds:
\begin{itemize}
\item $\overline{\gamma}_i\subset\overline{\gamma}_j$ and $\varepsilon_{\tilde{a}}(j)=-1$,
\item $\overline{\gamma}_i\supset\overline{\gamma}_j$ and $\varepsilon_{\tilde{a}}(i)=1$,
\item $\overline{a}_i\ne\overline{a}_j$ and $(\overline{a}_i,\overline{a}_j)$ is a well-ordered noncrossing couple.
\end{itemize}
\begin{oss}
\label{remark-aosevede}
We shall explain the geometric meaning behind these definitions. We have $i\prec_{\tilde{a}}'j$ if the second intersection of $\gamma_i$ with the curves drawn in red in Figures \ref{fig-ra1} and \ref{fig-ra2} (i.e. the point in which $\gamma_i$ ``exits" $R_1$ or $R_2$) is drawn lower in those figures than the second intersection of $\gamma_j$ with those curves. Meanwhile, we have $i\prec_{\tilde{a}}''j$ if the first intersection of $\gamma_i$ with the curves drawn in red in those figures (i.e. the point in which $\gamma_i$ ``enters" $R_1$ or $R_2$) is drawn lower in those figures than the first intersection of $\gamma_j$ with those curves. In the example case from Figure \ref{fig-ra2}, we have $1\prec_{\tilde{a}}'3\prec_{\tilde{a}}'2\prec_{\tilde{a}}'4$ and $1\prec_{\tilde{a}}''3\prec_{\tilde{a}}''4\prec_{\tilde{a}}''2$.
\end{oss}
\begin{oss}
\label{remark-preceq-depends}
Notice that $\preceq_{\tilde{a}}$ can be defined in terms of $\preceq_{\tilde{a}}'$, $\preceq_{\tilde{a}}''$ and $\varepsilon_{\tilde{a}}$.
\end{oss}
\begin{prop}
\label{oss-ff}
The relations $\preceq_{\tilde{a}}',\preceq_{\tilde{a}}''$ satisfy the following:
\begin{enumerate}
\item If $a_i\in W_1$ and $a_j\in W_2$, then $i\prec_{\tilde{a}}' j$ and $i\prec_{\tilde{a}}'' j$.
\item If $(\overline{a}_i,\overline{a}_j)$ is a noncrossing couple without inclusions, then we have $i\prec_{\tilde{a}}' j$ if and only if $i\prec_{\tilde{a}}'' j$.
\item If $(\overline{a}_i,\overline{a}_j)$ is not a noncrossing couple without inclusions, then we have $i\prec_{\tilde{a}}' j$ if and only if $j\prec_{\tilde{a}}'' i$.
\item We have $i\prec_{\tilde{a}}'' j$ if and only if $j-1\prec_{\tilde{a}}' i-1$.
\item If $i>1$ and $a_1,a_i\in W_1$, then $1\prec_{\tilde{a}}'' i$.
\item If $i>1$ and $a_1,a_i\in W_2$, then $1\succ_{\tilde{a}}'' i$.
\item If we have neither $a_i\le_T a_j$ nor $a_j\le_T a_i$, then we have $i\prec_{\tilde{a}}' j$ if and only if $i\prec_{\tilde{a}}'' j$.
\item If we have $a_i\le_T a_j$ or $a_j\le_T a_i$, then we have $i\prec_{\tilde{a}}' j$ if and only if $j\prec_{\tilde{a}}'' i$. 
\end{enumerate}
\end{prop}
\begin{proof}
The first three statements follow from the definitions. Statements 4,5,6 follow from Remark \ref{remark-aosevede} and the fact that $\gamma$ does not self-intersect. Statements 7,8 follow immediately from statements 2,3 respectively.

\begin{figure}[h]
\centering

\begin{tikzpicture}
\draw [red,very thick] (3,2.4) -- (-0.5,2.4) -- (-0.5,0.4) -- (3,0.4);
\draw [red,very thick] (3,-2.4) -- (-0.5,-2.4) -- (-0.5,-0.4) -- (3,-0.4);
\draw [fill] (-3,0) circle [radius=0.05];
\node[left] at (-3,0) {$x_0$};
\node[right] at (3,0) {$\gamma_g$};
\draw [fill] (-0.5,-1.4) circle [radius=0.05];
\node[left] at (-0.5,-1.4) {$A$};
\node at (1.25,-1.4) {$R_1$};
\draw [fill] (-0.5,1.8) circle [radius=0.05];
\node[left] at (-0.5,1.8) {$C$};
\draw [fill] (-0.5,1) circle [radius=0.05];
\node[left] at (-0.5,1) {$B$};
\node at (1.25,1.4) {$R_2$};
\draw [thick] (-3,-3) -- (-3,3) -- (3,3) -- (3,-3) -- (-3,-3);
\end{tikzpicture}
\caption{$\gamma_g$ is a noncrossing path representing $g$ such that $\gamma\subseteq\gamma_g$}
\label{fig-ossfalsa2}
\end{figure}
We shall show the proof of statement 6 in more detail as an example. With reference to Figure \ref{fig-ossfalsa2}, let $A$ be the point in which $\gamma_{i-1}$ ``exits" $R_1$, and let $B,C$ be, in some order, the points in which $\gamma_1,\gamma_i$ ``enter" $R_2$. If $B$ were the point in which $\gamma_1$ ``enters" $R_2$, there would be a segment of $\gamma$ going from $x_0$ to $B$, and one going from $A$ to $C$, and both of those segments would be contained in the closure of $\Int(\gamma_g)\setminus R_1\setminus R_2$. Those two segments would necessarily intersect, which contradicts the fact that $\gamma$ is noncrossing. Therefore, $B$ is the point in which $\gamma_i$ ``enters" $R_2$, which yields $i\prec_{\tilde{a}}'' 1$
\end{proof}
\begin{prop}
\label{prop-preceq-prime}
Let $\preceq',\preceq''$ be total order relations on $\{1,\dots,N\}$ that satisfy statements 1,4,5,6,7,8 from Proposition \ref{oss-ff}. Then, the relations $\preceq'$ and $\preceq_{\tilde{a}}'$ are equal, and so are $\preceq''$ and $\preceq_{\tilde{a}}''$.
\end{prop}
\begin{proof}
Let $1\le i<j\le N$. Assume we have neither $a_i,a_j\in W_1$ nor $a_i,a_j\in W_2$. Then, statement 1 from Proposition \ref{oss-ff} yields $i\prec'j$ if and only if $i\prec_{\tilde{a}}'j$, and $i\prec''j$ if and only if $i\prec_{\tilde{a}}''j$.

For the remainder of this proof, we shall assume that either $a_i,a_j\in W_1$ or $a_i,a_j\in W_2$ holds. This means that either $a_1,a_{j-i+1}\in W_1$ or $a_1,a_{j-i+1}\in W_2$ holds. So, either from statement 5 or from statement 6 from Proposition \ref{oss-ff}, we have $1\prec''j-i+1$ if and only if $1\prec_{\tilde{a}}''j-i+1$. Either from statement 7 or from statement 8 from Proposition \ref{oss-ff}, we have $1\prec'j-i+1$ if and only if $1\prec_{\tilde{a}}'j-i+1$.

We shall now prove by induction on $z$ that for each $z\in\{1,\dots,i\}$ we have $z\prec''j-i+z$ if and only if $z\prec_{\tilde{a}}''j-i+z$ and $z\prec'j-i+z$ if and only if $z\prec_{\tilde{a}}'j-i+z$. We already proved the case $z=1$, so we fix $z\in\{2,\dots,i\}$. By the induction hypothesis, we have $z-1\prec'j-i+z-1$ if and only if $z-1\prec_{\tilde{a}}'j-i+z-1$. Thus, by statement 4 from Proposition \ref{oss-ff} we have $z\prec''j-i+z$ if and only if $z\prec_{\tilde{a}}''j-i+z$. As before, either statement 7 or statement 8 from Proposition \ref{oss-ff} yields $z\prec'j-i+z$ if and only if $z\prec_{\tilde{a}}'j-i+z$. The statement follows from the case $z=i$.
\end{proof}
\begin{prop}
\label{oss-ff-2}
The function $\varepsilon_{\tilde{a}}:\{1,\dots,N\}\rightarrow\{1,-1\}$ satisfies the following:
\begin{enumerate}
\item $\varepsilon_{\tilde{a}}(1)=1$ if and only if $a_1\in W_1$
\item $\varepsilon_{\tilde{a}}(N)=1$ if and only if $a_N\in W_2$
\item For each $i\in\{2,\dots,N-1\}$, we have $\varepsilon_{\tilde{a}}(i)=-1$ if and only if $i-1\prec_{\tilde{a}}''i+1$
\end{enumerate}
\end{prop}
\begin{proof}
All statements follow from Remark \ref{remark-aosevede} and the fact that $\gamma$ does not self-intersect.
\end{proof}
We can now prove Lemma \ref{lemma-finalfg}, thus concluding the proof of Theorem \ref{thm-ele}.
\begin{proof}
We know from Proposition \ref{prop-preceq-prime} that the relations $\preceq_{\tilde{a}}',\preceq_{\tilde{a}}''$ only depend on $a$. From Proposition \ref{oss-ff-2}, we know that $\varepsilon_{\tilde{a}}$ only depends on $\preceq_{\tilde{a}}',\preceq_{\tilde{a}}''$, and thus only depends on $a$. The statement follows from Remark \ref{remark-aosevede}.
\end{proof}
\subsection{Proof of the main results}
Putting Theorems \ref{thm-focus}, \ref{thm-fg-hh} and \ref{thm-ele} together, we get the following:
\begin{teo}
\label{thm-final}
Assume that $(W_1,(s_1,\dots,s_k))$ and $(W_2,(s_{k+1},\dots,s_n))$ are well-stabilized, and that $(W_1,T_1,h_1)$ and $(W_2,T_2,h_2)$ are pan-transitive. Then, $(W,(s_{1},\dots,s_n))$ is well-stabilized, and $(W,T,h)$ is pan-transitive. In particular, we have that $\Art(W,S)$ is isomorphic to $\Art^*(W,T,h)$.
\end{teo}
\begin{oss}
\label{remark-freeprod}
From this and Remark \ref{remark-known-cases}, we have that if a Coxeter group $W$ can be written as a free product of spherical and affine groups, and we fix a product of Coxeter elements of those groups as the Coxeter element for $W$, then we have that the standard and the dual Artin groups are isomorphic via $\Psi$. It is worth noting that not all Coxeter elements of $W$ can be written in that form.
\end{oss}

Finally, we shall state a simpler version of Theorem \ref{thm-final} that deals with direct products instead of free products, and give a sketch of its proof.
\begin{teo}
\label{thm-direct}
Assume that $(W_1,(s_1,\dots,s_k))$ and $(W_2,(s_{k+1},\dots,s_n))$ are well-stabilized, and that $(W_1,T_1,h_1)$ and $(W_2,T_2,h_2)$ are pan-transitive. Define $W_\times\coloneqq W_1\times W_2$, $S_\times\coloneqq S_1\times\{1\}\cup \{1\}\times S_2$ and $T_\times\coloneqq T_1\times\{1\}\cup \{1\}\times T_2$. We then have that $(W_\times,((s_1,1),\dots,(s_k,1),$ $(1,s_{k+1}),\dots,(1,s_n)))$ is well-stabilized, and that $(W_\times,T_\times,(h_1,h_2))$ is pan-transitive. In particular, we have that $\Art(W_\times,S_\times)$ is isomorphic to $\Art^*(W_\times,T_\times,(h_1,h_2))$.
\end{teo}
\begin{proof}
Let $\Psi_1:\Art(W_1,S_1)\rightarrow \Art(W_1,T_1,h_1)$ and $\Psi_2:\Art(W_2,S_2)\rightarrow \Art(W_2,T_2,h_2)$ be the isomorphisms from Proposition \ref{stab-ele-isom}. It is natural to identify $\Art(W_\times,S_\times)$ and $\Art^*(W_\times,T_\times,(h_1,h_2))$ respectively with $\Art(W_1,S_1)\times \Art(W_2,S_2)$ and $\Art^*(W_1,T_1,h_1)\times \Art^*(W_2,T_2,h_2)$. We then see the product of the maps $\Psi_1,\Psi_2$ as a map $\Psi:\Art(W_\times,S_\times)\rightarrow \Art^*(W_\times,T_\times,(h_1,h_2))$, and $\Psi$ is an isomorphism since $\Psi_1,\Psi_2$ are isomorphism. Moreover, notice that $\Psi$ is the map from Proposition \ref{prop-def-psi}. Proposition \ref{isom-then-ws} then yields that $(W_\times,((s_1,1),\dots,(s_k,1),(1,s_{k+1}),\dots,(1,s_n)))$ is well-stabilized.

Let $a=(a_1,a_2)\in W_\times=W_1\times W_2$, and let $(t_1,\dots,t_N),(t_1',\dots,t_N')\in\Red_{T_\times}(a)$. There exist $\beta,\beta'\in\Br_N$ and $\ell\in\{0,\dots,N\}$ such that the first $\ell$ elements of the $N$-tuples $(u_1,\dots,u_N)\coloneqq\beta\cdot(t_1,\dots,t_N)$ and $(u_1',\dots,u_N')\coloneqq\beta'\cdot(t_1',\dots,t_N')$ belong to $T_1$ and form a reduced $T_1$-word for $a_1$, while the last $N-\ell$ elements of those $N$-tuples belong to $T_1$ and form a reduced $T_1$-word for $a_2$. Since $(W_1,T_1,h_1)$ and $(W_2,T_2,h_2)$ are pan-transitive, there exist $\tau_1\in\Br_\ell$ and $\tau_2\in\Br_{N-\ell}$ such that $(\tau_1,\tau_2)\cdot (u_1,\dots,u_N)=(u_1',\dots,u_N')$ (here we identify $\Br_\ell\times\Br_{N-\ell}$ with a subgroup of $\Br_N$ similarly to how it is described at the beginning of Section \ref{sect-fg}). Therefore, we have $(t_1',\dots,t_N')=(\beta')^{-1}(\tau_1,\tau_2)\beta\cdot (t_1,\dots,t_N)$, which proves that $(W_\times,T_\times,(h_1,h_2))$ is pan-transitive.
\end{proof}

\subsection{Applications}
\begin{prop}
\label{appl1}
Let $\Gamma$ be a graph with $\{1,\dots,N\}$ as its set of vertices, such that its connected components are complete graphs over sets of vertices of the form $\{a,a+1,\dots,b\}$. Let $W$ be the \emph{graph product} (defined in \cite{green}) of Coxeter groups $W_1, \dots, W_N$, each with a given Coxeter element and a factorization which satisfy the well-stabilized and pan-transitive conditions. Then, $W$ admits a set of simple reflections and a Coxeter element such that the standard and dual Artin groups are isomorphic to each other.
\end{prop}
\begin{proof}
$W$ can be written as the free product of the subgroups of $W$ generated by the groups corresponding to the vertices of the connected components of $\Gamma$. Each of those subgroups is the direct product of the groups corresponding to those vertices. Therefore, we can apply Theorems \ref{thm-direct} and \ref{thm-final}.
\end{proof}

Following \cite{charney}, given a graph $\Gamma$ with $\{1,\dots,N\}$ as its set of vertices, we say that the right-angle Coxeter group \emph{defined} by $\Gamma$ is the group:
\[\left<s_1, \dots, s_N | s_i s_j = s_j s_i \text{ for each }i,j\text{ connected by an edge in }\Gamma\right>\]

\begin{cor}
Let $\Gamma$ be a graph as in the statement of Proposition \ref{appl1}, and let $W$ be the right-angle Coxeter group defined by $\Gamma$. Then, $W$ admits a set of simple reflections and a Coxeter element such that the standard and dual Artin groups are isomorphic to each other.
\end{cor}
\begin{proof}
This is Proposition \ref{appl1} in the case where each of the groups $W_1, \dots, W_N$ is isomorphic to $\mathbb{Z}/2\mathbb{Z}$.
\end{proof}

\begin{prop}
\label{oss-raag3}
Let $(W,S)$ be the right-angle Coxeter system defined by a graph $\Gamma$, and suppose $(W,S)$ has rank 3. Then, there exists $(s_1,s_2,s_3)$ such that $\{s_1,s_2,s_3\}=S$, $(W,(s_1,s_2,s_3))$ is well-stabilized and $(W,s_1 s_2 s_3)$ is pan-transitive.
\end{prop}
\begin{proof}
If $\Gamma$ is not connected, then $W$ can be written as the free product of the subgroups of $W$ generated by the elements of $S$ corresponding to the connected components of $\Gamma$. Therefore, if we choose a suitable Coxeter element for $W$, we can apply Theorem \ref{thm-final}.

Meanwhile, if $\Gamma$ is connected, then there is a vertex of $\Gamma$ connected by an edge to each of the other vertices. That vertex corresponds to an element $s \in S$, and $W$ can therefore be written as the direct product of its subgroups generated by $s$ and $S\setminus\{s\}$. Therefore we can apply Theorem \ref{thm-direct}.
\end{proof}
As stated in Remark \ref{remark-known-cases}, it was already known that the standard and dual Artin groups are isomorphic to each other under the assumptions of Proposition \ref{oss-raag3}. However, that Proposition allows us to prove:

\begin{prop}
Let $(W,S)$ be the right-angle Coxeter system defined by a graph $\Gamma$, and suppose $(W,S)$ has rank 4. Then, there exists $(s_1,s_2,s_3,s_4)$ such that $\{s_1,s_2,s_3,s_4\}=S$, $(W,(s_1,s_2,s_3,s_4))$ is well-stabilized and $(W,s_1 s_2 s_3 s_4)$ is pan-transitive, except possibly if $\Gamma$ is the line graph.
\end{prop}
\begin{proof}
If $\Gamma$ is not connected, then $W$ can be written as the free product of the subgroups of $W$ generated by the elements of $S$ corresponding to the connected components of $\Gamma$. Therefore, if we choose a suitable Coxeter element for $W$, we can apply Theorem \ref{thm-final}, knowing that the assumptions of that theorem are satisfied because of Proposition \ref{oss-raag3}.

If there is a vertex of $\Gamma$ connected by an edge to each of the other vertices. That vertex corresponds to an element $s \in S$, and $W$ can therefore be written as the direct product of its subgroups generated by $s$ and $S\setminus\{s\}$. Therefore we can apply Theorem \ref{thm-direct}, again knowing that the assumptions of that theorem are satisfied because of Remark \ref{oss-raag3}.

Finally, if $\Gamma$ is connected, is not the line graph and has no vertex connected by an edge to each of the other vertices, then it admits the square graph as a subgraph, let $s,s'$ be the elements of $S$ corresponding to a pair of opposite vertices in that subgraph. $W$ can therefore be written as the direct product of its subgroups generated by $\{s,s'\}$ and $S\setminus\{s,s'\}$. Therefore we can apply Theorem \ref{thm-direct}.
\end{proof}

%\newpage

\printbibliography
\addcontentsline{toc}{section}{References}
\end{document}